\input ./style/arxiv-general.cfg
\documentclass[aap,MSNbibl,dvips]{arximspdf}
\makeatletter
   \@ifpackageloaded{graphicx}{}{\usepackage{graphicx}}
\makeatother
\usepackage{mathbh}

\doi{10.1214/15-AAP1120} 
\volume{26}
\issue{3}
\pubyear{2016}
\firstpage{1407}
\lastpage{1442}
\docsubty{FLA}

\makeatletter

\newcommand{\rrvert}{\vert}

\newcommand{\llvert}{\vert}
\def\mathds{\mathbh}

\renewcommand{\tilde}{\widetilde}

\newcommand{\R}{\mathbb{R}}
\newcommand{\N}{\mathbb{N}}
\newcommand{\E}{\mathbb{E}}

\newtheorem{theorem}{Theorem}
\newtheorem{corollary}[theorem]{Corollary}
\newtheorem{proposition}[theorem]{Proposition}
\newtheorem{lemma}[theorem]{Lemma}

\newproclaim{remark}[theorem]{Remark}
\newproclaim{ex}[theorem]{Example}

\newproclaim{definition}[theorem]{Definition}
\newtheorem{hyp}{Hypothesis}
\newcommand{\eqref}[1]{(\ref{#1})}
\makeatother

\begin{document}
\begin{frontmatter}

\title{Propagation of chaos for interacting particles subject to
environmental noise}
\runtitle{Propagation of chaos}

\begin{aug}
\author[A]{\fnms{Michele}~\snm{Coghi}\ead[label=e1]{michele.coghi@sns.it}} 
\and
\author[B]{\fnms{Franco}~\snm{Flandoli}\corref{}\ead[label=e2]{flandoli@dma.unipi.it}}
\runauthor{M. Coghi and F. Flandoli}
\affiliation{Scuola Normale Superiore and University of Pisa}
\address[A]{Scuola Normale Superiore\\
Piazza dei Cavalieri 7\\
56126 Pisa\\
Italy\\
\printead{e1}}
\address[B]{Dipartimento di Matematica\\
Universit\`{a} di Pisa\\
Piazza Bruno Pontecorvo 5\\
56127 Pisa\\
Italy\\
\printead{e2}}
\end{aug}

\received{\smonth{3} \syear{2014}}
\revised{\smonth{3} \syear{2015}}

%
\begin{abstract}
A system of interacting particles described by stochastic differential
equations is considered. As oppopsed  to the usual model, where the noise
perturbations acting on different particles are independent, here the
particles are subject to the same space-dependent noise, similar to the
(noninteracting) particles of the theory of diffusion of passive
scalars. We
prove a result of propagation of chaos and show that the limit PDE is
stochastic and of inviscid type, as opposed to the case when independent noises
drive the different particles.
\end{abstract}

%
\end{frontmatter}

\section{Introduction}\label{Introduction}

We prove a propagation of chaos result for the interacting particle
system in
$\mathbb{R}^{d}$ described by the equations
%
\begin{eqnarray}\label{SDE 1}
dX_{t}^{i,N}  =\frac{1}{N}\sum
_{j=1}^{N}K \bigl( X_{t}^{i,N}-X_{t}%
^{j,N}
\bigr) \,dt+\sum_{k=1}^{\infty}
\sigma_{k} \bigl( X_{t}^{i,N} \bigr) \circ
\,dB_{t}^{k},
\nonumber
\\[-8pt]
\\[-8pt]
\eqntext{i  =1,\ldots,N,}
\end{eqnarray}
where $K,\sigma_{k}:\mathbb{R}^{d}\rightarrow\mathbb{R}^{d}$, $k\in
\mathbb{N}%
$, are uniformly Lipschitz continuous and $ ( B^{k} ) _{k\in
\mathbb{N}}$ are
independent real-valued Brownian motions on a filtered probability space
$ ( \Omega,\mathcal{F},\mathcal{F}_{t},P ) $; the additional
assumption Hypothesis \ref{noise} will be imposed on $\sigma_{k}$'s, in
Section~\ref{Settings}. In \eqref{SDE 1}, we chose Stratonovich
stochastic integration since the final result, in Stratonovich form and
under Hypothesis \ref{noise}, is more clear and elegant. However, at
the price of additional terms, the
results hold for the It\^{o} case and under more general assumptions
(e.g., time-dependent $\sigma_k$); see Section~\ref{extensions
and variants}.

The classical propagation
of chaos framework considered in the literature deals with the system
%
\begin{eqnarray}\label{classical_system}
dX_{t}^{i,N}  =\frac{1}{N}\sum
_{j=1}^{N}K \bigl( X_{t}^{i,N}-X_{t}%
^{j,N}
\bigr) \,dt+dW_{t}^{i},
\nonumber
\\[-8pt]
\\[-8pt]
\eqntext{i  =1,\ldots,N,}
\end{eqnarray}
where $ ( W^{i} ) _{i\in\mathbb{N}}$ are independent
$\mathbb
{R}%
^{d}$-valued Brownian motions; see, for instance,~\cite{Sz}.
Unlike this
classical case, in (\ref{SDE 1}) \textit{the same} space-dependent
delta-correlated-in-time noise $v ( t,x ) $, formally given
by
\[
v ( t,x ) =\sum_{k=1}^{\infty}
\sigma_{k} ( x ) \frac{dB_{t}^{k}}{dt}
\]
acts on each particle. This type of space correlated noise was
introduced in physics to describe
small scale motion in a turbulent fluid,
as in the famous Kraichnan
model
of the sixties. The physical intuition in this case, for equation
\eqref
{SDE 1}, is that the
particles are embedded in a turbulent fluid with velocity $v(t, x)$.
Each particle is subject to the transport effect of the fluid and to
the motion caused by the interaction with the other particles. Among
other examples, we may also think of the case of smoothed point
vortices (think of relatively large scale vortex structures in ocean or
atmosphere), subject to the transport effect of each other (the
interaction) and of a background, small scale, turbulent perturbation.
Instead of considering all fluid scales as a whole, described by classical
equations of fluid dynamics, one could try, phenomenologically, to
separate the
large scale vortex structures from the small scale more irregular fluctuations
and consider the small scales modeled independently a priori, and the vortices
just influencing each other and influenced by the small scales  without
feedback on small scales. In such an example, to fit with the
assumptions of model \eqref{SDE 1}, we have to assume that the
interaction between vortices is described by a smoothed Biot--Savart
kernel  since the singularity of the true Biot--Savart kernel
introduces additional difficulties which cannot be handled with the
techniques of this paper.
On the other hand,
the
more classical model \eqref{classical_system} is more suitable when
each particle has its own internal origin of randomness (like certain
living organisms) or the external sources of randomness can be
considered to be totally uncorrelated at the scale of the particles, like
for very light macroscopic particles interacting with the molecules of
a gas.

If the covariance of the noise is suitably
concentrated (see Hypothesis \ref{noise} in Section~\ref{Settings}),
the random field $v ( t,x ) $ is poorly
space-correlated, except at very short distances, and thus particles which
occupy sufficiently distant positions are subject to almost independent noise,
a fact that makes the two systems (\ref{SDE 1}) and (\ref
{classical_system}) not so
different when the collection of particles is sufficiently sparse.

However, in the limit when $N\rightarrow\infty$, the behavior is completely
different. Let $ ( X^{i} ) _{i\in\mathbb{N}}$ be a
sequence of
i.i.d. random vectors in $\mathbb{R}^{d}$ with law $\mu_{0}$; assume
that the
families $ ( B_{\cdot}^{k} ) _{k\in\mathbb{N}}$ [$ (
W^{i} ) _{i\in\mathbb{N}}$ for equation (\ref
{classical_system})] and
$ ( X^{i} ) _{i\in\mathbb{N}}$ are independent and take
$X_{0}^{i,N}=X^{i}$ as initial conditions for system (\ref{SDE 1}). Denote
by $S_{t}^{N}$ the empirical measure defined as
%
\begin{equation}
\label{empirical} S_{t}^{N}=\frac{1}{N}\sum
_{i=1}^{N}\delta_{X_{t}^{i,N}}.
\end{equation}
The random probability measure $S_{0}^{N}$ converges weakly to $\mu
_{0}$ in
probability. In both cases of equations (\ref{SDE 1}) and (\ref
{classical_system}), one can
prove [cf. \cite{Sz} for case (\ref{classical_system}) and the present
paper for case
(\ref{SDE 1})] that $S_{t}^{N}$ converges weakly, in probability, to a
probability measure $\mu_{t}$. 
However, in case (\ref{classical_system}), $\mu_{t}$ is deterministic,
the weak
convergence of $S_{t}^{N}$ to $\mu_{t}$ is understood in probability with
respect to both initial conditions and noise, and $\mu_{t}$ is a
distributional solution of the nonlinear equation
\[
\frac{\partial\mu_{t}}{\partial t}+\operatorname{div} ( b_{\mu
_{t}}%
\mu_{t} ) =\frac{1}2\Delta\mu_{t},
\]
where, for a generic probability measure $\nu$, the vector field
$b_{\nu
}:\mathbb{R}^{d}\rightarrow\mathbb{R}^{d}$ is defined as
\[
b_{\nu} ( x ) =\int_{\mathbb{R}^{d}}K ( x-y ) \nu ( dy ).
\]
On the contrary, in case (\ref{SDE 1}), $\mu_{t}$ is a random
probability measure and, under the particular assumptions of Section~\ref{assumptions}, it satisfies in the distributional sense the
stochastic PDE
%
\begin{equation}
\label{SPDE Ito form} d\mu_{t}+\operatorname{div} ( b_{\mu_{t}}
\mu_{t} ) \,dt+\sum_{k=1}^{\infty}
\operatorname{div} ( \sigma_{k} \mu_{t} ) \circ
dB_{t}^{k}=0%
\end{equation}
and the weak convergence of $S_{t}^{N}$ to $\mu_{t}$ is understood in
probability only with respect to the initial conditions. In Section~\ref{assumptions}, we give the It\^{o} form of this stochastic partial
differential equation and in Section~\ref{extensions and variants} we
show the modifications when we start from \eqref{SDE 1} in
It\^{o} form or when the assumptions on $\sigma_k$ are more general
than those of Section~\ref{assumptions}.

The main result of this paper is the following theorem, by
which one can relate the convergence of the empirical measure of the
system with the convergence of the empirical measure of the initial conditions.

\begin{theorem}\label{teorema_convergenza}
Let $T>0$ and assume Hypothesis \ref{noise}, given in Section~\ref{Settings}, on the noise. There exists a constant $\tilde C_T>0$ such that
\[
\E\bigl[W_1\bigl(\mu,S_t^N\bigr)\bigr]\leq
\tilde{C}_T\E\bigl[ W_1\bigl(\mu_0,S^N_0
\bigr)\bigr],
\]
where $W_1$ is the Wasserstein distance (see Definition~\ref
{Wasserstein_distance}).
\end{theorem}

In Section~\ref{sezione_convergenza} we give a more precise statement
of Theorem~\ref{teorema_convergenza}, as well as a short discussion on
recent results on quantitative estimates on the rate of convergence of
$S_0^N$ to $\mu_0$ which can be applied in our model.

From Theorem~\ref{teorema_convergenza} we deduce a \textit
{conditional} propagation of
chaos result: Conditional to $ ( B^{k} ) _{k\in\mathbb{N}}$,
the particles tend to be
independent as $N\rightarrow\infty$. One can find other works in
literature dealing with conditional propagation of chaos, but referring
to different objects and in different contexts. In \cite{CCLLV} and
\cite{Hauray_Mischler_2014}, the authors treat propagation of chaos
conditionally to produce measures on the Kac's sphere and in the latter
are given quantitative estimates. In other works, the conditionality is
given with respect to the $\sigma$-field of the permutable events; see,
for example, \cite{Zheng} and \cite{Dermoune}.

The precise statement about conditional propagation of chaos in this
work is given by the following theorem.

\begin{theorem}\label{teorema_2}
\label{main theorem} Let
$\mathcal{F}_{t}^{B}$ be the filtration associated to $ (
B^{k} )
_{k\in\mathbb{N}}$. We suppose that the noise satisfies Hypothesis
\ref
{noise} in both equations \eqref{SDE 1} and \eqref{SPDE Ito form}.
There exists a random measure-valued solution $\mu_{t}$ of
equation (\ref{SPDE Ito form}) such that
\[
\lim_{N\rightarrow\infty}E \bigl[ \bigl\llvert \bigl\langle
S_{t}^{N} 
,\phi \bigr\rangle- \langle
\mu_{t},\phi \rangle\bigr\rrvert \bigr] =0
\]
for all $\phi\in C_{b} ( \mathbb{R}^{d} ) $.

Moreover, given $r\in\mathbb{N}$ and $\phi_{1},\ldots,\phi_{r}\in
C_{b} (
\mathbb{R}^{d} ) $, we have
\[
\lim_{N\rightarrow\infty}E \bigl[ \phi_{1} \bigl(
X_{t}^{1,N} \bigr) \cdots\phi_{r} \bigl(
X_{t}^{r,N} \bigr) |\mathcal {F}_{t}^{B}
\bigr] =%
{ \prod_{i=1}^{r}}
 \langle\mu_{t},\phi_{i} \rangle
\]
in $L^{1} ( \Omega ) $.
\end{theorem}

In particular, for every $r\in\mathbb{N}$ and $\phi\in C_{b} (
\mathbb{R}^{d} ) $, $\lim_{N\rightarrow\infty}E [ \phi
 (
X_{t}^{r,N} ) |\mathcal{F}_{t}^{B} ] = \langle\mu
_{t}%
,\phi \rangle$, namely the conditional law of $X_{t}^{r,N}$ given
$\mathcal{F}_{t}^{B}$ converges weakly to $\mu_{t}$. We can also prove
the following.

\begin{theorem}\label{ConvergenzaL1}
Given $\mu_{t}$ as in Theorem~\ref{main theorem} and $r\in\mathbb
{N} $,
if $X_{t}$ is the unique strong solution of the SDE
\[
dX_{t}=b_{\mu_{t}} ( X_{t} ) \,dt+\sum
_{k=1}^{\infty}\sigma _{k} ( X_{t} )
\,dB_{t}^{k},\qquad X_{0}=X_{0}^{r},
\]
where the noise satisfies Hypotesis \ref{noise}, then
\[
\lim_{N\rightarrow\infty}E \bigl[ \bigl\llvert X_{t}^{r,N}-X_{t}
\bigr\rrvert \bigr] =0.
\]
Moreover, $\mu_t$ is a version of the conditional law of $X_t$ with
respect to $\mathcal{F}^B_t$, namely
\[
\langle\mu_t,\phi\rangle\in\E \bigl[\phi(X_t)|\mathcal
{F}_t^B \bigr]
\]
for every $\phi\in C_b^\infty(\R^d)$.
\end{theorem}

The result is similar to the case of a \textit{deterministic}
environment acting on the
particles, which could be modeled by the equations
%
\begin{eqnarray}
\frac{dX_{t}^{i,N}}{dt}  =\frac{1}{N}\sum_{j=1}^{N}K
\bigl( X_{t}%
^{i,N}-X_{t}^{j,N}
\bigr) +v \bigl( t,X_{t}^{i,N} \bigr),
\nonumber\\
\eqntext{i  =1,\ldots,N.}
\end{eqnarray}
As  shown by \cite{Do}, this system satisfies a propagation of chaos
property with the limit deterministic inviscid PDE
\[
\frac{\partial\mu_{t}}{\partial t}+\operatorname{div} ( b_{\mu
_{t}}%
\mu_{t} ) \,dt+\operatorname{div} \bigl( v ( x ) \mu _{t}
\bigr) =0.
\]
Also some technical steps of our proof are strongly inspired by \cite{Do}.
Moreover, with a different proof and partially a different purpose,
some of
the technical steps about existence and (especially) stability
results for measure-valued stochastic equations have been proved before by
\cite{Pi,Pi2,Sk}.

We do not treat here a number of additional interesting questions that are
postponed to future works, like: (i) the fact that $\mu_{t}$ should
have a
density with respect to Lebesgue measure if this is assumed for $\mu_{0}$;
(ii) the uniqueness of solutions to the SPDE (\ref{SPDE}) (which seems
to be
true in some class of integrable functions when $\mu_{0}$ has an integrable
density, but it is less clear in spaces of measure-valued solutions);
(iii) possible generalizations to non-Lipschitz continuous interation kernel
$K$. In particular, the problem of propagation of chaos for system
(\ref{SDE 1}) when $K ( x ) =\frac{x^{\perp}}{\llvert
x\rrvert ^{2}}$, corresponding to point vortices in 2D inviscid fluids,
has been posed by \cite{tre} and seems to be a challenging question.


In Section~\ref{Settings}, we  give some information about the
settings in which we study the problem. Section~\ref{Well_SPDE} is
devoted to the study of existence and uniqueness of equation \eqref
{SPDE Ito form} using its It\^{o} version. Finally, in Section~\ref{sezione_convergenza} we study the convergence and propagation of
chaos results.

\section{Precise setting of the problem}\label{Settings}

\subsection{Assumptions on the noise}\label{assumptions}
We will now state the assumptions which we will consider on the noise.
Recall that $\sigma_k:\R^d\to\R^d$ is a vector field, for every
$k\in\N$.

\begin{hyp}\label{noise}
\textup{(i)} $\sigma_k:\R^d\to\R^d$ are measurable and satisfy\break $\sum_{k=1}^{\infty}\llvert \sigma_{k} (x ) \rrvert
^{2}<+\infty$, for every $x\in\mathbb{R}^{d}$.
\begin{enumerate}[(iii)]
\item[(ii)] $\sigma_k$ is a $C^2$ divergence free vector fields, that is,
\[
\operatorname{div}\sigma_k=0\qquad \forall k\geq1.
\]
%
Define
the matrix-valued function $Q:\mathbb{R}^{d}\times\mathbb{R}^{d}%
\rightarrow\mathbb{R}^{d\times d}$ as
%
\begin{equation}
\label{definition_Q} Q^{ij} ( x,y ) :=\sum_{k=1}^{\infty}
\sigma_{k}^{i} ( x ) \sigma_{k}^{j} (
y ).
\end{equation}
\item[(iii)] With a little abuse of notation, there exists a
function $Q:\mathbb{R}^{d}\rightarrow\mathbb{R}^{d\times d}$ such that:
\begin{enumerate}[(a)]
\item[(a)]$Q( x,y) =Q( x-y)$ [space homogeneity of the random field
$\varphi
 ( t,x )
=\sum_{k=1}^{\infty}\sigma_{k} ( x ) B_{t}^{k}$];
\item[(b)]$Q (0 ) =\mathrm{Id}$;
\item[(c)]$Q ( \cdot ) $ is of class $C^{2}$ with second
derivatives uniformly bounded in the euclidean norm of $\R^{d\times
d}$, that is, $\sup_{x\in\R^d}|\partial^2_{x_ix_j}Q(x)|<+\infty$.
Here, we
are using the Hilbert--Schmidt norm on the space of the matrices.
\end{enumerate}
\end{enumerate}
\end{hyp}

One can find examples of this model in several references, for example,
\cite{DFV} and~\cite{Ku90}. We recall here the most important
properties of this type of noise and we give an explicit example.

\begin{remark}
Under the previous assumptions, we have
%
\begin{equation}
\sum_{k=1}^{\infty}\bigl\llvert
\sigma_{k} ( x ) -\sigma _{k} ( y ) \bigr\rrvert
^{2}\leq L_{\sigma}^{2}\llvert x-y\rrvert
^{2} \qquad\mbox{for all }x,y\in\mathbb{R}^{d}\label{Lipschitz property}%
\end{equation}
for some constant $L_{\sigma}>0$. Indeed,
\[
\sum_{k=1}^{\infty}\bigl\llvert
\sigma_{k} ( x ) -\sigma _{k} ( y ) \bigr\rrvert
^{2}=2\operatorname{Tr} \bigl( Q ( 0 ) \bigr) -2\operatorname{Tr} \bigl( Q ( x-y ) \bigr).
\]
The function $f ( z ) =\operatorname{Tr} ( Q ( z )  )
$ has
the property $f ( -z ) =f ( z ) $, hence from the
identity $2f ( z ) =f ( z ) +f ( -z )
$ and
Taylor development of both $f ( z ) $ and $f ( -z
) $ we
get $2f ( z ) =2f ( 0 ) + \langle
D^{2}f (
0 ) z,z \rangle+o ( \llvert  z\rrvert
^{2} ) $
which implies\break $\sum_{k=1}^{\infty}\llvert \sigma_{k} (
x )
-\sigma_{k} ( y ) \rrvert ^{2}\leq C_{1}\llvert
x-y\rrvert ^{2}$ if $\llvert  x-y\rrvert \leq1$, for a suitable
constant $C_{1}>0$. When $\llvert  z\rrvert >1$ we have $f (
z ) \leq C_{2}\llvert  z\rrvert ^{2}$ for a suitable constant
$C_{2}>0$, because $Q ( \cdot ) $ has bounded second derivative.
Hence, $\sum_{k=1}^{\infty}\llvert \sigma_{k} ( x )
-\sigma
_{k} ( y ) \rrvert ^{2}\leq C_{2}\llvert  x-y\rrvert
^{2}$ when $\llvert  x-y\rrvert >1$. This proves
(\ref{Lipschitz property}) with $L_{\sigma}^{2}=\max ( C_{1}%
,C_{2} ) $.
\end{remark}

It is also important to notice that the covariance function $Q$ can be
given first. Indeed Theorem~4.2.5 of \cite{Ku90} states that any matrix
valued function $Q:(x,y)\to Q(x,y)$ satisfying \eqref{Lipschitz
property} can be expressed in the form \eqref{definition_Q}.
A very common example of this kind of noise is the isotropic random
field, which we  present now.

\begin{ex}\label{ex1}
Let $d\geq2$ and $f\in L^1(\R_+)$ such that $\int_{\R^d}|y|^2
f(|y|) \,d
y<+\infty$. Given $\pi(y)$ a $d\times d$ matrix defined as
\[
\pi(y)=(1-p)\mathrm{Id}_d+|y|^{-2}(pd-1)y\otimes y \qquad\mbox{for } y\in
\R ^d, p\in[0,1],
\]
we consider
\[
Q(x)=\int_{\R^d}e^{iy\cdot x}\pi(y)f\bigl(|y|\bigr) \,d y,\qquad x\in
\R^d.
\]
It is easy to see that property (iii)(a) is satisfied. Property
(iii)(c) is true after a renormalization in $L^1$ of $f$ and (iii)(c)
can be verified with a straightforward computation.
\end{ex}

\begin{remark}
A strong solution of system \eqref{SDE 1} is a continuous process
$ (
X^{1,N}, \ldots,X^{N,N} ) $, adapted to $ ( \mathcal
{F}_{t}^{B} )
_{t\geq0}$, such that
\[
P \Biggl( \sum_{k=1}^{\infty}\int
_{0}^{T}\bigl\llvert \sigma_{k} \bigl(
X_{t}^{i,N} \bigr) \bigr\rrvert ^{2}\,dt<\infty
\Biggr) =1
\]
for every $i=1,\ldots,N$ (so that the series of stochastic integrals converge
in probability) and identity \eqref{SDE 1} holds in the integral sense.
But\break $%
\sum_{k=1}^{\infty}\llvert \sigma_{k} ( X_{t}^{i,N} )
\rrvert ^{2}=\operatorname{Tr} ( Q ( 0 )  ) =d$, hence the
sum of
stochastic integrals in equation \eqref{SDE 1} always converges, even
in mean square.
\end{remark}

\subsection{It\^o formulation}\label{Ito formulation}
In the \hyperref[Introduction]{Introduction}, for the benefit of interpretation, we have
formulated the interacting particle system and the limit SPDE both in
Stratonovich form. However, for  the sake of rigor and mathematical
simplicity, it is convenient to work in the corresponding It\^{o} form.
Under Hypothesis \ref{noise}, the interacting particle
system in It\^{o} form is
%
\begin{eqnarray}\label{SDE 2}
dX_{t}^{i,N}  =\frac{1}{N}\sum
_{j=1}^{N}K \bigl( X_{t}^{i,N}-X_{t}%
^{j,N}
\bigr) \,dt+\sum_{k=1}^{\infty}
\sigma_{k} \bigl( X_{t}^{i,N} \bigr)
\,dB_{t}^{k},
\nonumber
\\[-8pt]
\\[-8pt]
\eqntext{i  =1,\ldots,N}
\end{eqnarray}
and the SPDE \eqref{SPDE Ito form} in It\^{o} form is
%
\begin{equation}
\label{SPDE} d\mu_{t}+\operatorname{div} ( b_{\mu_{t}}
\mu_{t} ) \,dt+\sum_{k=1}^{\infty
}
\operatorname{div} \bigl( \sigma_{k} ( x ) \mu_{t} \bigr)
\,dB_{t}^{k}=\frac{1}2\Delta\mu_t,
\end{equation}
which will be interpreted in weak form in Definition~\ref{def
measure-valued solution} below. At the rigorous level, these are
the equations to which the statements of the \hyperref[Introduction]{Introduction} apply.
Let us motivate the fact that \eqref{SDE 2} and \eqref{SPDE} correspond
to \eqref{SDE 1} and \eqref{SPDE Ito form} under Hypothesis~\ref
{noise}. This correspondence can be made rigorous but it requires
[especially for \eqref{SPDE Ito form}] proper definitions of solutions
and a number of details. If we accept that \eqref{SDE 1} and \eqref
{SPDE Ito form} are given only for interpretation ad the rigorous
setup is given by \eqref{SDE 2} and \eqref{SPDE}, an heuristic proof
of their equivalence is sufficient.
The correspondence between \eqref{SDE 1} and \eqref{SDE 2} is due to
the fact that the Stratonovich
integral $\int_{0}^{t}\sigma_{k} ( X_{s}^{i,N} )
\circ dB_{s}^{k}$ is equal to
\[
\int_{0}^{t}\sigma_{k} \bigl(
X_{s}^{i,N} \bigr) \,dB_{s}^{k}+
\frac
{1}{2}%
\int_{0}^{t} ( D
\sigma_{k}\cdot\sigma_{k} ) \bigl( X_{s}^{i,N}
\bigr) \,ds
\]
%
(see \cite{Ku90}) where $ ( D\sigma_{k}\cdot\sigma_{k} )
_{i} (
x ) =\sum_{j=1}^{d}\sigma_{k}^{j} ( x ) \partial_{j}
\sigma_{k}^{i} ( x ) $. This correction term vanishes thanks
to the assumption
\[
\operatorname{div}\sigma_{k}=0\qquad\mbox{for each }k\in
\mathbb{N}%
\]
[it is natural if we interpret $v ( t,x ) $ as the velocity
field of an incompressible
fluid] along with the assumptions on $Q$ made above. Indeed,
\begin{eqnarray*}
0= \Biggl( \sum_{j=1}^{d}
\partial_{j} \Biggr) Q^{ij} ( 0 ) &=&\sum
_{k=1}^{\infty}\sum_{j=1}^{d}
\partial_{j} \bigl( \sigma _{k}^{j} ( x )
\sigma_{k}^{i} ( x ) \bigr) \\
&=&\sum
_{k=1}^{\infty}%
\sum
_{j=1}^{d}\sigma_{k}^{j} ( x )
\partial_{j}\sigma_{k} 
^{i} ( x ).
\end{eqnarray*}
Therefore, the Stratonovich and It\^{o} formulations coincide for the
interacting particle system.

Let us discuss now the correspondence between \eqref{SPDE Ito form} and
\eqref{SPDE}.
The Stratonovich integral $\int_{0}^{t}\operatorname{div} ( \sigma
_{k} (
x ) \mu_{s} ) \circ dB_{s}^{k}$ is formally equal to (one should
write all terms applied to test functions)
\[
\int_{0}^{t}\operatorname{div} \bigl(
\sigma_{k} ( x ) \mu _{s} \bigr) \,dB_{s}^{k}-
\frac{1}{2}\int_{0}^{t}\operatorname{div} \bigl(
\sigma_{k} ( x ) \operatorname{div} \bigl( \sigma_{k} ( x ) \mu
_{s} \bigr) \bigr) \,ds
\]
[the second term, with heuristic language, is initially given by\break $\frac
{1}{2}%
\int_{0}^{t}\operatorname{div} ( \sigma_{k} ( x ) \,d
\langle
\mu
, B^{k} \rangle_{s} ) $ where $ \langle\mu
, B^{k} \rangle_{s}$ is the mutual quadratic covariation; then
we use
again equation \eqref{SPDE Ito form} to compute $d \langle\mu
,B^{k} \rangle_{s}$
and get $d \langle\mu,B^{k} \rangle_{s}=\operatorname{div} (
\sigma
_{k} ( x ) \mu_{s} ) \,ds$]. Now we see that
%
\begin{eqnarray}\label{new identity Ito Strat}%
&&\sum_{k=1}^{\infty}\operatorname{div} \bigl(
\sigma_{k} ( x ) \operatorname{div} \bigl( \sigma_{k} ( x )
\mu_{s} \bigr) \bigr)
\nonumber
\\[-8pt]
\\[-8pt]
\nonumber
&&\qquad=\sum_{\alpha,\beta=1}^{d}
\partial_{\alpha}\partial_{\beta} \bigl( Q^{\alpha\beta} ( x,x )
\mu_{s} \bigr) -\operatorname {div} \Biggl( \Biggl( \sum
_{k=1}^{\infty}D\sigma_{k}\cdot
\sigma_{k} \Biggr) \mu _{s} \Biggr),
\end{eqnarray}
where $D\sigma_{k}\cdot\sigma_{k}$ is the vector field with components
\[
( D\sigma_{k}\cdot\sigma_{k} ) ^{\alpha}=\sum
_{\beta
=1}^{d} \bigl( \partial_{\beta}
\sigma_{k}^{\alpha} \bigr) \sigma_{k}^{\beta}.
\]
Indeed,
\begin{eqnarray*}
\sum_{k=1}^{\infty}\operatorname{div} \bigl(
\sigma_{k} ( x ) \operatorname{div} \bigl( \sigma_{k} ( x )
\mu_{s} \bigr) \bigr) &=&\sum_{k=1}^{\infty}
\sum_{\alpha,\beta=1}^{d}\partial_{\alpha
}
\bigl( \sigma_{k}^{\alpha} ( x ) \partial_{\beta} \bigl(
\sigma _{k}%
^{\beta} ( x ) \mu_{s} \bigr)
\bigr)
\\
&=&\sum_{k=1}^{\infty}\sum
_{\alpha,\beta=1}^{d}\partial_{\alpha
}\partial
_{\beta
} \bigl( \sigma_{k}^{\alpha} ( x )
\sigma_{k}^{\beta
} ( x ) \mu_{s} \bigr)\\
&&{} -\sum
_{k=1}^{\infty}\sum_{\alpha,\beta
=1}%
^{d}
\partial_{\alpha} \bigl( \bigl( \partial_{\beta}\sigma
_{k}^{\alpha
} \bigr) ( x ) \sigma_{k}^{\beta}
( x ) \mu _{s} \bigr)
\end{eqnarray*}
and $\sum_{k=1}^{\infty}\sigma_{k}^{\alpha} ( x ) \sigma
_{k}%
^{\beta} ( x ) =Q^{\alpha\beta} ( x,x ) $. Moreover,
%
\begin{eqnarray}\label{second new identity}%
\sum_{k=1}^{\infty} \bigl( D
\sigma_{k} ( x ) \cdot\sigma _{k} ( x ) \bigr)
^{\alpha}&=&\sum_{k=1}^{\infty}\sum
_{\beta
=1}^{d} \bigl( \partial_{\beta}
\sigma_{k}^{\alpha} ( x ) \bigr) \sigma _{k}^{\beta}
( x )
\nonumber
\\[-8pt]
\\[-8pt]
\nonumber
&=&\sum_{\beta=1}^{d}\partial_{\beta}%
Q^{\alpha\beta}
( x,x ) -\sum_{k=1}^{\infty}\sigma
_{k}^{\alpha
} ( x ) \operatorname{div}\sigma_{k} ( x
) .
\end{eqnarray}
In view of the next section, we stress that until now we have not used
Hypothesis \ref{noise}. Under Hypothesis \ref{noise}, we have
$Q^{\alpha
\beta} ( x,x )
=\delta_{\alpha\beta}$ and $\operatorname{div}\sigma_{k}=0$, hence
$\sum_{k=1}^{\infty} ( D\sigma_{k} ( x ) \cdot\sigma
_{k} (
x )  ) ^{\alpha}=0$ for all $\alpha=1,\ldots,d$, and finally
\[
\sum_{k=1}^{\infty}\operatorname{div} \bigl(
\sigma_{k} ( x ) \operatorname{div} \bigl( \sigma_{k} ( x )
\mu_{s} \bigr) \bigr) =\Delta\mu_{s}.
\]
Therefore, the It\^{o} formulation of equation \eqref{SPDE Ito form} is
\eqref{SPDE}.

\subsection{Extensions and variants}\label{extensions and variants}

As we remarked in the \hyperref[Introduction]{Introduction}, we chose to work under Hypothesis
\ref{noise}  since it leads to particularly simple and elegant
equations and
relations between It\^{o} and Stratonovich formulations. However, all the
results hold in more general cases, some of which we discuss here.

Assume $u,\sigma_{k}: [ 0,T ] \times\mathbb
{R}^{d}\rightarrow
\mathbb{R}^{d}$, $k\in\mathbb{N}$, are measurable vector fields such
that, for
some constants $C,L>0$
\begin{eqnarray*}
\bigl\llvert u ( t,x ) \bigr\rrvert ^{2}+\sum
_{k=1}^{\infty
}\bigl\llvert \sigma_{k} ( t,x )
\bigr\rrvert ^{2}&\leq& C \bigl( 1+\llvert x\rrvert ^{2}
\bigr),
\\
\bigl\llvert u ( t,x ) -u ( t,y ) \bigr\rrvert ^{2} 
+\sum
_{k=1}^{\infty}\bigl\llvert \sigma_{k}
( t,x ) -\sigma _{k} ( t,y ) \bigr\rrvert ^{2}&\leq& L\llvert
x-y\rrvert ^{2}%
\end{eqnarray*}
for all $x,y\in\mathbb{R}^{d}$ and all $t\in [ 0,T ] $.
Under these
conditions, always with $K$ Lipschitz continuous, consider the system of
equations in It\^{o} form
%
\begin{eqnarray}\label{SDE 3}
dX_{t}^{i,N}  =\frac{1}{N}\sum
_{j=1}^{N}K \bigl( X_{t}^{i,N}-X_{t}%
^{j,N}
\bigr) \,dt+u \bigl( t,X_{t}^{i,N} \bigr) \,dt+\sum
_{k=1}^{\infty
}%
\sigma_{k} \bigl(
t,X_{t}^{i,N} \bigr) \,dB_{t}^{k},
\nonumber
\\[-8pt]
\\[-8pt]
\eqntext{i  =1,\ldots,N.}
\end{eqnarray}
Set
\begin{eqnarray*}
Q_{t}^{\alpha\beta} ( x,y ) & :=&\sum_{k=1}^{\infty
}
\sigma _{k}^{\alpha} ( t,x ) \sigma_{k}^{\beta}
( t,y ),
\\
a^{\alpha\beta} ( t,x ) & :=&Q_{t}^{\alpha\beta} ( x,x ).
\end{eqnarray*}
All results of the present paper hold true in this case with the corresponding
SPDE given by
%
\begin{eqnarray}\label{SPDE general}%
&&d\mu_{t}+\operatorname{div} \bigl( ( b_{\mu_{t}}+u ) \mu
_{t} \bigr) \,dt+\sum_{k=1}^{\infty}
\operatorname{div} ( \sigma _{k}\mu _{t} )
\,dB_{t}^{k}
\nonumber
\\[-8pt]
\\[-8pt]
\nonumber
&&\qquad=\frac{1}{2}\sum
_{\alpha,\beta
=1}^{d}\partial _{\alpha
}
\partial_{\beta} \bigl( a^{\alpha\beta} ( t,\cdot ) \mu_{t}%
^{N}
\bigr) \,dt
\end{eqnarray}
(to be interpreted in weak form similar to Definition~\ref{def
measure-valued solution} below). The connection between these two
equations can be seen informally in a few lines by applying It\^{o}
formula to
$\phi ( X_{t}^{i,N} ) $, with $\phi\in C_{c}^{\infty
} (
\mathbb{R}^{d} ) $; the result is that $S_{t}^{N}$ satisfies
\begin{eqnarray*}
d \bigl\langle S_{t}^{N},\phi \bigr\rangle&= &\bigl\langle
S_{t}^{N},\nabla \phi\cdot ( b_{\mu_{t}}+u ) \bigr
\rangle \,dt+\sum_{k=1}%
^{\infty} \bigl
\langle S_{t}^{N},\nabla\phi\cdot\sigma_{k} ( t,
\cdot ) \bigr\rangle \,dB_{t}^{k}\\
&&{}+ \Biggl\langle
S_{t}^{N},\frac{1}%
{2}\sum
_{\alpha,\beta=1}^{d}a^{\alpha\beta} ( t,\cdot )
\partial_{\alpha}\partial_{\beta}\phi \Biggr\rangle \,dt,
\end{eqnarray*}
which is the weak formulation of the SPDE \eqref{SPDE general} above.

\begin{remark}
Assuming a suitable differentiability of $\sigma_{k} ( t,\cdot
) $
in the $t$ variable, we may rewrite the SPDE (\ref{SPDE general}) in
Stratonovich form. We keep this remark at heuristic level, to avoid
unnecessary details. As in the previous section, the Stratonovich integral
$\int_{0}^{t}\operatorname{div} ( \sigma_{k} ( s,x )
\mu
_{s} ) \circ dB_{s}^{k}$ is equal to the It\^{o} integral $\int_{0}
^{t}\operatorname{div} ( \sigma_{k} ( s,x ) \mu
_{s} )
\,dB_{s}^{k}$ plus the correction term
%
\begin{equation}
\tfrac{1}{2} \bigl[ \operatorname{div} \bigl( \sigma_{k} (
\cdot,x ) \mu_{\cdot} \bigr),B_{\cdot}^{k} \bigr]
_{t}%
.\label{covariation sect 2.3}%
\end{equation}
Now, $\sigma_{k} ( t,x ) \mu_{t}$ formally satisfies the identity
(by It\^{o}'s formula)
\[
d \bigl( \sigma_{k} ( t,x ) \mu_{t} \bigr) =
\frac
{\partial
\sigma_{k}}{\partial t} ( t,x ) \mu_{t}\,dt+\sigma_{k} ( t,x ) \,d
\mu_{t}%
\]
hence only the term
\[
-\sigma_{k} ( t,x ) \sum_{k^{\prime}=1}^{\infty
}
\operatorname {div}%
 \bigl( \sigma_{k^{\prime}} ( t,x )
\mu_{t} \bigr) \,dB_{t} 
^{k^{\prime}}%
\]
contributes to the quadratic covariation (\ref{covariation sect 2.3}), which
is thus equal (as in the previous section) to
\[
-\frac{1}{2}\int_{0}^{t}
\operatorname{div} \bigl( \sigma_{k} ( s,x ) \operatorname{div} \bigl(
\sigma_{k} ( s,x ) \mu _{s} \bigr) \bigr) \,ds.
\]
From identity (\ref{new identity Ito Strat}), where now $Q^{\alpha
\beta
} ( x,x ) $ is replaced by $a^{\alpha\beta} (
t,x ) $,
we get that $\mu_{t}$ satisfies (in weak form) the Stratonovich equation
%
\begin{equation}\quad
d\mu_{t}=-\operatorname{div} \bigl( ( b_{\mu_{t}}+u ) \mu
_{t} \bigr) \,dt-\sum_{k=1}^{\infty}
\operatorname{div} \bigl( \sigma _{k} ( t,\cdot ) \mu_{t}
\bigr) \circ dB_{t}^{k}+\mathcal{D} ( t,\cdot )
\mu_{t}\,dt,\label{general Strat SPDE}%
\end{equation}
where the first-order differential operator $\mathcal{D} (
t,x
) $
is given by
\[
\mathcal{D}f:=\frac{1}{2}\operatorname{div} \Biggl( \sum
_{k=1}^{\infty}%
D\sigma_{k}\cdot
\sigma_{k} f \Biggr).
\]
\end{remark}

\begin{remark}
The Stratonovich reformulation (\ref{general Strat SPDE}) reveals that the
true nature of the SPDE (\ref{SPDE general}) is not parabolic but of a
first-order equation, informally speaking of hyperbolic type.
\end{remark}

If we start from the beginning with the Stratonovich equation,
\[
dX_{t}^{i,N}=\frac{1}{N}\sum
_{j=1}^{N}K \bigl( X_{t}^{i,N}-X_{t}^{j,N}
\bigr) \,dt+u \bigl( t,X_{t}^{i,N} \bigr) \,dt+\sum
_{k=1}^{\infty}\sigma _{k} \bigl(
t,X_{t}^{i,N} \bigr) \circ dB_{t}^{k}%
\]
in place of (\ref{SDE 3}), we may rewrite it in the It\^{o} form
\begin{eqnarray*}
dX_{t}^{i,N} & =&\frac{1}{N}\sum
_{j=1}^{N}K \bigl( X_{t}^{i,N}-X_{t}%
^{j,N}
\bigr) \,dt+u \bigl( t,X_{t}^{i,N} \bigr) \,dt
\\
&&{} +\sum_{k=1}^{\infty}\sigma_{k}
\bigl( t,X_{t}^{i,N} \bigr) \,dB_{t}^{k}%
+
\frac{1}{2}\sum_{k=1}^{\infty} ( D
\sigma_{k}\cdot\sigma _{k} ) \bigl( t,X_{t}^{i,N}
\bigr) \,dt,
\end{eqnarray*}
where $ ( D\sigma_{k}\cdot\sigma_{k} ) ^{\alpha}=\sum_{\beta=1}
^{d}\partial_{\beta}\sigma_{k}^{\alpha}\sigma_{k}^{\beta}$. This
is the
case because the
correction term of the $\alpha$-component is
\[
\frac{1}{2}\sum_{k=1}^{\infty}\,d \bigl[
\sigma_{k}^{\alpha} \bigl( \cdot,X_{\cdot}^{i,N}
\bigr) ,B_{t}^{k} \bigr] _{t}=\frac
{1}{2}
\sum_{k=1}^{\infty}\nabla\sigma_{k}^{\alpha}
\bigl( t,X_{t}^{i,N} \bigr) \cdot\sigma_{k} \bigl(
t,X_{t}^{i,N} \bigr) \,dt
\]
since, under suitable differentiability assumptions on $\sigma_{k}$,
we may
apply It\^{o}'s formula to $\sigma_{k}^{\alpha} (
t,X_{t}^{i,N}
) $
and see that for the quadratic covariation $ [ \sigma_{k}^{\alpha
} (
\cdot,\break X_{\cdot}^{i,N} ) ,B_{t}^{k} ] _{t}$ only the
following term
[part of $\nabla\sigma_{k}^{\alpha} ( t,X_{t}^{i,N} )
\cdot
dX_{t}^{i,N}$] matters:
\[
\nabla\sigma_{k}^{\alpha} \bigl( t,X_{t}^{i,N}
\bigr) \cdot\sum_{k^{\prime
}=1}^{\infty}
\sigma_{k^{\prime}} \bigl( t,X_{t}^{i,N} \bigr)
\,dB_{t}%
^{k^{\prime}}.
\]
Thus we see that under appropriate regularity and summability (in $k$)
properties on $\sigma_{k}$, we may transform the Stratonovich equation into
the It\^{o} one (\ref{SDE 3}) and apply the previous result. The additional
drift
%
\begin{equation}
\frac{1}{2}\sum_{k=1}^{\infty} ( D
\sigma_{k}\cdot\sigma _{k} ) ( t,x ) \label{additional drift}%
\end{equation}
appears in the It\^{o} formulation.

Finally, we have seen that two annoying correction terms appear in the
computations above, namely $\mathcal{D} ( t,\cdot ) \mu
_{t}$ in
the SPDE (\ref{general Strat SPDE}) and the additional drift
(\ref{additional drift}). Both are related to passages from It\^{o} to
Stratonovich forms. Both of them are equal to zero if we assume
\[
\sum_{k=1}^{\infty}D\sigma_{k}\cdot
\sigma_{k}=0.
\]
Similar to \eqref{second new identity}, this can be rewritten as
\[
\sum_{\beta=1}^{d}\partial_{\beta}a^{\alpha\beta}
( t,x ) -\sum_{k=1}^{\infty}
\sigma_{k}^{\alpha}\operatorname{div}\sigma_{k}=0.
\]
A sufficient condition thus is the pair of assumptions
\begin{eqnarray*}
&&a^{\alpha\beta} ( t,x ) \qquad\mbox{independent of }x,
\\
&&\operatorname{div}\sigma_{k}=0\qquad \mbox{for every }k,
\end{eqnarray*}
which are part of Hypothesis \ref{noise}.

\subsection{Some definitions}

Recall the definition of the empirical measure $S_t^N:=\frac{1}N\sum_{i=1}^N\delta_{X_t^{i,N}}$, which can be used, as we did in the
\hyperref[Introduction]{Introduction}, to rewrite the drift coefficient as $b_{S_t^N}(x)=K\ast
S^N_t(x)=\frac{1}N\sum_{j=1}^NK(x-X_t^{j,N})$. We can thus write equation
\eqref{SDE 2}, for $i=1,\ldots,N$, as
\[
d X^{i,N}_t=b_{S_t^N}\bigl(X_t^{i,N}
\bigr) \,d t+\sum_{k=1}^\infty\sigma
_k\bigl(X^{i,N}_t\bigr) \,d B_t^k.
\]
If we take a test function $\phi\in C^2_b(\R^d)$ and we apply It\^{o}'s
formula, from the assumptions on $Q$ it follows, for $ i=1,\ldots,N$,
\begin{eqnarray*}
d \phi\bigl(X_t^{i,N}\bigr)&=& \biggl[\nabla\phi
\bigl(X_t^{i,N}\bigr)\cdot b_{S_t^N}
\bigl(X_t^{i,N}\bigr) +\frac{1}2\Delta\phi
\bigl(X^{i,N}_t\bigr) \biggr] \,d t \\
&&{}+\sum
_{k=1}^\infty\nabla\phi\bigl(X_t^{i,N}
\bigr)\cdot\sigma_k\bigl(X_t^{i,N}\bigr) \,d
B_t^k,
\end{eqnarray*}
which becomes, adding over $N$ and dividing by $N$,
\[
\bigl\langle S_t^N,\phi\bigr\rangle= \biggl[\bigl\langle
S_t^N,\nabla\phi\cdot b_{S_t^N}\bigr\rangle+
\frac{1}2\bigl\langle S_t^N,\Delta\phi\bigr\rangle
\biggr] \,d t+ \sum_{k=1}^\infty\bigl\langle
S_t^N,\nabla\phi\cdot\sigma_k\bigr\rangle \,d
B_t^k.
\]
Hence, $S_{t}^{N}$ is a measure-valued solution of equation (\ref%
{SPDE Ito form}), in the sense of Definition~\ref{def1} below.

We define now the space over which we will study equation \eqref{SPDE
Ito form}.

\begin{definition}\label{Wasserstein_distance}
%
%
$(\mathcal{P}_1(\R^d),W_1)$ is the space of probability measures
$\mu_0$ on $\R^d$ with finite first moment, that is,
\[
\|\mu_0\|:=\int_{\R^d} \,d \mu_0=1,\qquad
M_1(\mu_0):=\int_{\R^d}|x| \,d \mu
_0(x)<\infty
\]
endowed with the 1-Wasserstein metric defined as
\[
W_1(\nu_0,\mu_0)=\inf_{m\in\Gamma(\mu_0,\nu_0)}
\int_{\R
^{2\,d}}|x-y| m( d x, d y), \qquad\mu_0,
\nu_0\in\mathcal{P}_1\bigl(\R^d\bigr).
\]
Here, $\Gamma(\mu_0,\nu_0)$ is the set of the finite measures on $\R
^{2d}$ with first and second marginals equal respectively to $\mu_0$
and $\nu_0$, namely
\begin{eqnarray*}
&&\hspace*{-4pt}\Gamma(\mu_0,\nu_0)\\
&&\hspace*{-4pt}\qquad=\bigl\{m \in\mathcal{P}_1
\bigl(\R^{2d}\bigr) : m\bigl(A\times\R ^d\bigr)=
\mu_0(A), m\bigl(\R^d\times A\bigr)=\nu_0(A),
\forall A\in\mathcal{B}\bigl(\R ^d\bigr)\bigr\}.
\end{eqnarray*}
$\mathcal{S}$ will be the space of the stochastic processes
taking values on $(\mathcal{P}_1(\R^d),W_1)$,
\[
\mu: [0,T]\times\Omega\to\mathcal{P}_1\bigl(\R^d\bigr)
\]
such that $\E [\sup_{t\in[0,T]}\int_{\R^d}|x| \,d\mu
_t(x)
]<\infty$ and $\langle\mu_t,\phi\rangle$ is $\mathcal
{F}_t$-adapted for
every test function $\phi\in C_b^\infty(\R^d)$.
We endow $\mathcal{S}$ with the following distance:
\[
 d_{\mathcal{S}}(\mu,\nu):=\E \Bigl[\sup
_{t\in[0,T]}W_1(\mu_t,\nu _t)
\Bigr], 
\]
where $\mu=(\mu_t)_{t\in[0,T]},\nu=(\nu_t)_{t\in[0,T]}\in
\mathcal{S}$.
\end{definition}


\begin{remark}
The metric space $(\mathcal{P}_1(\R^d),W_1)$ has been well studied in
optimal transportation theory and extensive results on it can be found
in the literature, (see, e.g., \cite{AGS}). In particular, this space
is complete and separable (Proposition~7.1.5 of \cite{AGS}).
Hence, follows from standard arguments that $(\mathcal{S},d_{\mathcal
{S}})$ is also a complete metric space.
\end{remark}

\begin{hyp}\label{Initial Condition} Concerning the initial condition
$\mu_0 : \Omega\to\mathcal{P}_1(\R^d)$ of equation \eqref{SPDE Ito
form} we shall always assume that:
\begin{longlist}[(ii)]
\item[(i)]$\mu_0$ is $\mathcal{F}_0$-measurable;
\item[(ii)]$\E [\int_{\R^d}|x| \,d \mu_0(x) ]< \infty$.
\end{longlist}
\end{hyp}

For every $\mu_0$ that satisfies the previous hypothesis, we call
$\mathcal{S}_{\mu_0}$ the set of $\mu\in\mathcal{S}$ such that
$\mu
|_{t=0}=\mu_0$.

\begin{definition}\label{def1}
\label{def measure-valued solution}A family $ \{
\mu_{t} ( \omega );t\geq0,\omega\in\Omega \} $ of
random probability measures taking value in $\mathcal{P}_1(\R^d)$ is a
measure-valued solution of equation (\ref{SPDE Ito form}) if:
\begin{longlist}[(i)]
\item[(i)] for all $\phi\in C_{b} ( \mathbb{R}^{d} ) $, $
\langle\mu_{t},\phi \rangle$ is an adapted process with a
continuous version,
\item[(ii)]
for all $\phi\in C_{b}^{2} ( \mathbb{R}^{d} ) $
\begin{eqnarray*}
\langle\mu_{t},\phi \rangle&=& \langle\mu _{0},\phi
\rangle+\int_{0}^{t} \langle\mu_{s},b_{\mu_{s}}
\cdot \nabla \phi \rangle \,ds+\frac{1}{2}\int_{0}^{t}
\langle\mu _{s},\Delta \phi \rangle \,ds
\\
&&{}+\sum_{k=1}^{\infty}\int_{0}^{t}
\langle\mu_{s},\sigma _{k}\cdot \nabla\phi \rangle
\,dB_{s}^{k}.
\end{eqnarray*}
\end{longlist}
\end{definition}

\begin{remark}
Notice that the infinite sum in the previous equation converges under
our assumptions. Indeed, if $\phi\in C_b^2(\R^d)$, it holds, by It\^{o}
isometry and Jensen inequality,
\begin{eqnarray*}
\E \Biggl[\Biggl\llvert \sum_{k=1}^\infty\int
_0^t\langle\mu_s,\sigma
_k\cdot\nabla \phi\rangle \,d B_s^k\Biggr
\rrvert ^2 \Biggr]&=&\E \Biggl[\sum_{k=1}^\infty
\int_0^t\langle\mu_s,
\sigma_k\cdot\nabla\phi\rangle^2 \,d s \Biggr]\\
&\leq&\E \Biggl[
\sum_{k=1}^\infty\int_0^t
\bigl\langle\mu_s,|\sigma_k\cdot \nabla\phi
|^2\bigr\rangle \,d s \Biggr].
\end{eqnarray*}
Now, by the assumptions on $\sigma_k$, we have
\begin{eqnarray*}
\sum_{k=1}^\infty\bigl|\sigma_k(x)
\cdot\nabla\phi(x)\bigr|^2&\leq&\sum_{k=1}^\infty
\bigl|\nabla\phi(x)\bigr|^2\bigl|\sigma_k(x)\bigr|^2=\bigl|\nabla
\phi(x)\bigr|^2\sum_{k=1}^\infty \bigl|
\sigma_k(x)\bigr|^2\\
&\leq& C\bigl|\nabla\phi(x)\bigr|^2<+\infty.
\end{eqnarray*}
\end{remark}

\section{Well posedness of the stochastic PDE}\label{Well_SPDE}
In this chapter, we study the well posedness of equation \eqref{SPDE
Ito form}, and thus we prove the following.

\begin{theorem}\label{maintheo}
Let $T\geq0$ and $\mu_0:\Omega\to\mathcal{P}_1(\R^d)$ be as in
Hypothesis \ref{Initial Condition}. There exists a unique solution
$\mu
=(\mu_t)_{t\in[0,T]}$ of equation \eqref{SPDE Ito form} in the sense of
Definition~\ref{def1} starting from $\mu_0$ and defined up to time $T$,
that can be seen as the only fixed point of the operator \eqref{contr}
defined below.
\end{theorem}

We have already seen that the empirical measure $S_t^N$ defined in
\eqref{empirical} satisfies in the distributional sense \eqref{SPDE Ito
form} for every test function $\phi$, moreover it is a probability
measure with finite first moment and the process
\[
\bigl\langle S_t^N,\phi\bigr\rangle=\frac{1}N
\sum_{i=1}^N\phi\bigl(X_t^{i,N}
\bigr)
\]
is $\mathcal{F}_t$-adapted. This is true since the processes
$X^{i,N}_t$ are solutions of the SDE \eqref{SDE 2}, and hence are
adapted and continuous.
Hence, the empirical measure $S_t^N$ satisfies \eqref{SPDE Ito form} in
the sense of Definition~\ref{def1}.

\subsection{Stochastic Liouville equation}
In order to investigate the solutions of equation \eqref{SPDE Ito
form}, we first want to study what happens when the drift coefficient
does not depend on the solution but it is instead a priori defined (but
random). We hence consider the following stochastic differential equation:
%
\begin{eqnarray}
\label{liouville} %
d X_t &=&
b(t,X_t) \,d t + \sum_k^\infty
\sigma_k(X_t) \,d B_t^k,
\nonumber
\\[-8pt]
\\[-8pt]
\nonumber
X_0&=&x\in\R^d,
\end{eqnarray}
where the $\sigma_k$'s are defined as before. Here, $b=b(t,x,\omega)$
is an $\mathcal{F}_t$-adapted process, continuous in $(t,x)$, which satisfies:
\begin{itemize}
\item$b$ Lipschitz continuous in $x$ uniformly in $(t,\omega)$, with
Lipschitz constant $L_b$, not depending on $\omega$ and $t$, that is,
\[
\bigl|b(t,x,\omega)-b(t,y,\omega)\bigr|\leq L_b |x-y| \qquad\forall x,y\in\R
^d, \forall t\in\R, \mathbb{P}\mbox{-a.s.}
\]
\item For every fixed $w$, $b$ has linear growth in $x$ uniformly $t$,
that is,
\[
\bigl|b(t,x,\omega)\bigr|\leq c_1|x|+c_2(\omega)\qquad \forall x\in
\R^d, \forall t\in\R, \mbox{ for } \mathbb{P}\mbox{-a.e. } \omega,
\]
where $c_1\in\R$ and $c_2(\omega)$ is a random variable such that
$\E
[|c_2(\omega)|]<\infty$.
\end{itemize}
%

By classical results on SDEs (see, e.g., \cite{Ku90}), this equation
admits a unique solution $X_t=X(t,x,\omega)$ which is continuous in
time. Moreover, taking into account the following lemma, it follows
from Kolmogorov continuity theorem that there exists a modification of
$X(t,x)$ which is continuous in $x$. It is also jointly continuous in
$(t,x)$ by Kolmogorov theorem for processes taking values in Banach
spaces, precisely in the space $C([0,T];\R^d)$. This results on
continuity of the stochastic flow of equation \eqref{liouville} can be
found in the literature as in \cite{Ku90}. However, we want to stress
in the following the dependence on the different parameters and outline
more explicitly the constants.

We define now some constants depending on the coefficients $b$ and
$\sigma_k$ of the problem, which we will use in the following results.
For a fixed real number $p\geq1$, we call $C_p$ the constat which
appears in the Burkholder--Davis--Gundy theorem. Moreover, for $t>0$
and $p\geq1$, we define
%
\begin{equation}
\label{first_constant} C(p,t):=C_pT^{{1} /{(2p)}}
L_\sigma+T^{{1}/p}L_b.
\end{equation}
Finally, for a fixed $T>0$, let $n\in\mathbb{N}$ be the minimum such
that $C (p,(T/n) )<1$, so that we can define
%
\begin{equation}
\label{second_constant} C_{p,T}:= \bigl(1-C \bigl(p,(T/n) \bigr)
\bigr)^{-np}.\vadjust{\goodbreak}
\end{equation}
From our choice of $n\in\mathbb{N,}$ this last constant is well defined
and depending only on $T,p$ and the coefficients of problem \eqref{liouville}.

\begin{lemma}\label{limsol}
Let $p\geq1$, $T\geq0$ and let $X(t,x)$ be a solution of equation
\eqref{liouville} up to time $T$. Then
%
\begin{equation}
\label{prekol} \E \Bigl[\sup_{t\in[0,T]}\bigl\llvert X(t,x)-X
\bigl(t,x'\bigr)\bigr\rrvert ^p \big|\mathcal
{F}_0 \Bigr]\leq C_{p,T}\bigl|x-x'\bigr|^p,
\end{equation}
where the constant $C_{p,T}$ is defined in \eqref{second_constant}.
\end{lemma}


%
\begin{pf}
Let $n\in\N$ be the minimum such that $C (p,(T/n) )<1$, where
$C(\cdot,\cdot)$ is defined in \eqref{first_constant}.
Now we divide the temporal interval $[0,T]$ in $n$ subintervals. We set
$X^{(0)}(t,x)=x$ and we call $X^{(m)}$, for $m=1,\ldots,n$, the solution to
\begin{eqnarray*} d X_t &= &b(t,X_t) \,d t +
\sum_k^\infty\sigma_k(X_t)
\,d B_t^k,
\\
X_{({(m-1)}/n) T}&=&X^{(m-1)}\biggl(\frac{m-1}{n}T,x\biggr)
\end{eqnarray*}
on the interval $[\frac{m-1}nT,\frac{m}n T]$.
We prove by induction that, for every $m=1,\ldots,n$,
%
\begin{eqnarray}
\label{induzione}&& E \Bigl[\sup_{t\in[({(m-1)}/{n})T,({m}/{n})T]}\bigl\llvert X^{(m)}
(t,x )-X^{(m)} \bigl(t,x' \bigr)\bigr\rrvert
^p \big|\mathcal {F}_0 \Bigr]^{{1}/p}
\nonumber
\\[-8pt]
\\[-8pt]
\nonumber
&&\qquad\leq
\frac{|x-x'|}{ (1-C (p,(T/n) ) )^{m}}.
\end{eqnarray}
It follows from the uniqueness of solution of the stochastic
differential equations that the solution $X_t$ of equation \eqref
{liouville} coincides on each interval $[\frac{m-1}nT,\break \frac{m}n T]$
with the process $X_t^{(m)}$. The thesis follows noting that the worst
constant in \eqref{induzione} appears when $m=n$ and it coincides with
$C_{p,t}$.
%

\textit{Step} 1. Now we prove \eqref{induzione} for $m=1$. By a
triangular inequality, we get
\begin{eqnarray*}
&&\E \Bigl[\sup_{t\in[0,(T/n)]}\bigl\llvert X^{(1)}(t,x)-X^{(1)}
\bigl(t,x'\bigr)\bigr\rrvert ^p \big|\mathcal{F}_0
\Bigr]^{{1}/p}\\
&&\qquad\leq\bigl|x-x'\bigr|
\\
&&\qquad\quad{}+\E \biggl[\sup_{t\in[0,(T/n)]}\biggl\llvert \int_0^tb
\bigl(s,X^{(1)}(s,x)\bigr)-b\bigl(s,X^{(1)}\bigl(s,x'
\bigr)\bigr) \,d s\biggr\rrvert ^p \Big|\mathcal {F}_0
\biggr]^{{1}/p}
\\
&&\qquad\quad{}+\E \biggl[\sup_{t\in[0,(T/n)]}\biggl\llvert \int_0^t
\sum_k\sigma _k\bigl(X^{(1)}(s,x)
\bigr)-\sigma_k\bigl(X^{(1)}\bigl(s,x'\bigr)
\bigr) \,d B_s^k\biggr\rrvert ^p \Big|\mathcal
{F}_0 \biggr]^{{1}/p}.
\end{eqnarray*}
In order to estimate this, we first notice that, by the Lipschitz
continuity of $b$ one can get
\begin{eqnarray*}
&&\E \biggl[\sup_{t\in[0,(T/n)]}\biggl\llvert \int_0^tb
\bigl(s,X^{(1)}(s,x)\bigr)-b\bigl(s,X^{(1)}\bigl(s,x'
\bigr)\bigr) \,d s\biggr\rrvert ^p \Big|\mathcal {F}_0
\biggr]^{{1}/p}\\
&&\qquad\leq \bigl({(T/n)} \bigr)^{{1}/p}L_b
\E \Bigl[\sup_{t\in[0,(T/n)]}\bigl\llvert X^{(1)}(t,x)-X^{(1)}
\bigl(t,x'\bigr)\bigr\rrvert ^p \big|\mathcal
{F}_0 \Bigr]^{{1}/p}.
\end{eqnarray*}
Now, using the conditional Burkholder--Davis--Gundy inequality
(Proposition~\ref{CBDG}), we obtain
\begin{eqnarray*}
&&\E \biggl[\sup_{t\in[0,(T/n)]}\biggl\llvert \int_0^t
\sum_k\sigma _k\bigl(X^{(1)}(s,x)
\bigr)-\sigma_k\bigl(X^{(1)}\bigl(s,x'\bigr)
\bigr) \,d B_s^k\biggr\rrvert ^p \Big|\mathcal
{F}_0 \biggr]^{{1}/p}
\\
&&\qquad\leq C_p \E \biggl[ \biggl(\int_0^{(T/n)}
\sum_k\bigl\llvert \sigma _k
\bigl(X^{(1)}(s,x)\bigr)-\sigma_k\bigl(X^{(1)}
\bigl(s,x'\bigr)\bigr)\bigr\rrvert ^2 \,d s
\biggr)^{{p}/2}\Big |\mathcal{F}_0 \biggr]^{{1}/p}
\\
&&\qquad\leq C_p \bigl({(T/n)} \bigr)^{{1}/ {(2p)}}
L_\sigma\E \Bigl[\sup_{t\in
[0,(T/n)]}\bigl\llvert
X^{(1)}(t,x)-X^{(1)}\bigl(t,x'\bigr)\bigr\rrvert
^p\big |\mathcal{F}_0 \Bigr]^{{1}/p}.
\end{eqnarray*}
We have hence proved the base step of the induction.

\textit{Step} 2. Now we suppose \eqref{induzione} true for $m$ and we
prove it for $m+1$. First, thanks to a triangular inequality we obtain
\begin{eqnarray*}
&&\E \Bigl[\sup_{t\in[({m}/{n})T,({(m+1)}/{n})T]}\bigl\llvert X^{(m+1)}(t,x)-X^{(m+1)}
\bigl(t,x'\bigr)\bigr\rrvert ^p \big|\mathcal{F}_0
\Bigr]^{{1}/p}
\\
&&\qquad\leq\E \biggl[\biggl\llvert X^{(m)} \biggl(\frac{m}{n}T,x
\biggr)-X^{(m)} \biggl(\frac{m}{n}T,x' \biggr)\biggr
\rrvert ^p \Big|\mathcal{F}_0 \biggr]^{{1}/p}
\\
&&\qquad\quad{}+\E \biggl[\sup_{t\in[({m}/{n})T,({(m+1)}/{n})T]}\biggl\llvert
\int_{({m}/{n})T}^tb
\bigl(s,X^{(m+1)}(s,x)\bigr)\\
&&\qquad\quad{}-b\bigl(s,X^{(m+1)}\bigl(s,x'
\bigr)\bigr) \,d s\biggr\rrvert ^p \Big|\mathcal{F}_0
\biggr]^{{1}/p}
\\
&&\qquad\quad{}+\E \biggl[\sup_{t\in[({m}/{n})T,({(m+1)}/{n})T]}\biggl\llvert
\int_{({m}/{n})T}^t
\sum_k\sigma_k\bigl(X^{(m+1)}(s,x)
\bigr)\\
&&\qquad\quad{}-\sigma_k\bigl(X^{(m+1)}\bigl(s,x'\bigr)
\bigr) \,d B_s^k\biggr\rrvert ^p \Big|
\mathcal{F}_0 \biggr]^{{1}/p.}
\end{eqnarray*}
Now, as in step 1, we use the Lipschitz property of $b$ and $\sigma_k$
and Lemma~\ref{CBDG} to get
\begin{eqnarray*}
&&\E \Bigl[\sup_{t\in[({m}/{n})T,({(m+1)}/{n})T]}\bigl\llvert
 X^{(m+1)}(t,x)-X^{(m+1)}
\bigl(t,x'\bigr)\bigr\rrvert ^p\big |\mathcal{F}_0
\Bigr]^{{1/}p}\\
&&\qquad\leq\frac{\E [\llvert X^{(m)} (({m}/{n})T,x
)-X^{(m)} (({m}/{n})T,x' )\rrvert ^p |\mathcal
{F}_0
]^{{1}/p}}{ (1-C (p,(T/n) ) )^p}.
\end{eqnarray*}
Now estimate the right-hand side using \eqref{induzione} for $m$, and
we conclude this last step,
\begin{eqnarray*}
&&\E \Bigl[\sup_{t\in[({m}/{n})T,({(m+1)}/{n})T]}\bigl\llvert X^{(m+1)}(t,x)-X^{(m+1)}
\bigl(t,x'\bigr)\bigr\rrvert ^p \big|\mathcal{F}_0
\Bigr]^{{1}/p}\\
&&\qquad\leq\frac{|x-x'|}{ (1-C (p,(T/n) )
)^{m}}\frac{1}{ (1-C (p,(T/n) ) )}.
\end{eqnarray*}
\upqed\end{pf}

Using the continuous version in $x$ of the solution of equation \eqref
{liouville}, we are going to define a solution for equation \eqref{SPDE
Ito form} in the case in which the drift coefficient is fixed. This is
shown in the following proposition. The push forward described in the
next statement has to be understood $\omega$-wise: for a.e. $\omega$
and for each $t\in[0,T]$, we take the initial measure $\mu_0(\omega
)=\mu
_0(\omega, d x)$ and we consider its image measure (or push forward)
under the continuous map $x\mapsto X(t,x,\omega)$, denoted by $\mu
_t(\omega)$ or $\mu_t(\omega, d x)$.

\begin{proposition}\label{propuno}
Given $\mu_0$ which satisfies Hypotesis \ref{Initial Condition}, the
push forward of $\mu_0$ with respect to the solution of \eqref
{liouville} namely
\[
\label{lsol} \mu_t(\omega)=X(t,\cdot,\omega)_\#\mu_0(
\omega)
\]
solves the following equation in the sense of Definition~\ref{def1}:
\[
\label{wliouville} \cases{ %
\displaystyle d \mu_t= -
\operatorname{div} (b \mu_t ) \,d t - \sum_{k=0}^\infty
\operatorname {div} (\sigma_k \mu_t ) \,d B_t^k+
\frac{1}2\Delta\mu_t,
\vspace*{2pt}\cr
\mu_{t|_{t=0}}=\mu_0. }
\]
\end{proposition}

\begin{pf}
First, notice that $\mu\in\mathcal{S}$. By definition, for every
$t\in
[0,T]$ and $\mathbb{P}$-a.s., $\mu_t$ is a finite and positive measure.
We show that the first moment of $\mu_t$ is finite,
%
\begin{eqnarray}
 \E \biggl[\int_{\R^d}|x| \,d \mu_t(x)
\biggr] &= &\E \biggl[\int_{\R
^d}|X_t| \,d
\mu_0(x) \biggr]
\nonumber
\\
&\leq& \E \biggl[\int_{\R^d}|x| \,d \mu_0(x)
\biggr]\label{25}
\\
&&{}+ \E \biggl[\int_{\R^d}\int_0^t\bigl|b
\bigl(X(s,x)\bigr)\bigr| \,d s\, d \mu_0(x) \biggr]\label
{26}
\\
&&{}+ \E \Biggl[ \int_{\R^d}\Biggl\llvert \int
_0^t\sum_k^\infty
\sigma _k\bigl(X(s,x)\bigr) \,d B_s^k\Biggr
\rrvert \,d\mu_0 (x) \Biggr].\label{27}
\end{eqnarray}
It follows from the choice of $\mu_0$ that \eqref{25} is finite. We can
bound \eqref{26} if we notice that the Lipschitz continuity assumption
on $b$ implies $|b(x)|\leq1+|x|$, which gives
%
\begin{equation}
\label{momento_finito_1} \E \biggl[\int_{\R^d}\int_0^t\bigl|b
\bigl(X(s,x)\bigr)\bigr| \,d s\, d \mu_0(x) \biggr]\leq CT+C\int
_0^t\E \biggl[\int_{\R^d}|x|
\,d \mu_s(x) \biggr] \,ds.
\end{equation}
In order to bound \eqref{27}, we use Propositions \ref
{ScambioCondizionale} and \ref{CBDG} and we do the following:
%
\begin{eqnarray}
\label{momento_finito_2} &&\E \Biggl[\E \Biggl[ \int_{\R^d}\Biggl\llvert
\int_0^t\sum_{k=0}^\infty
\sigma _k\bigl(X(s,x)\bigr) \,d B_s^k\Biggr
\rrvert \,d\mu_0(x) \Big|\mathcal{F}_0 \Biggr] \Biggr]\nonumber\\
&&\qquad=\E
\Biggl[ \int_{\R^d}\E \Biggl[\Biggl\llvert \int
_0^t\sum_{k=0}^\infty
\sigma _k\bigl(X(s,x)\bigr) \,d B_s^k\Biggr
\rrvert\Big |\mathcal{F}_0 \Biggr] \,d\mu_0 (x) \Biggr]
\nonumber
\\[-8pt]
\\[-8pt]
\nonumber
&&\qquad\leq C\E \Biggl[ \int_{\R^d}\E \Biggl[\int
_0^t\sum_{k=0}^\infty
\bigl|\sigma _k\bigl(X(s,x)\bigr)\bigr|^2 \,d s\Big |
\mathcal{F}_0 \Biggr]^{{1}/2} \,d\mu_0 (x)
\Biggr]
\\
&&\qquad \leq C\sqrt{T}.\nonumber
\end{eqnarray}
Here, we used $\sum_{k=0}^\infty\llvert \sigma_k (X(s,x)
)\rrvert ^2<+\infty$. Taking into account \eqref{momento_finito_1} and
\eqref
{momento_finito_2}, we can apply the Gronwall lemma to deduce that the
first moment of $\mu_t$ is finite for every $t$.
Let us stress a detail. In order to apply Proposition~\ref
{ScambioCondizionale} of the \hyperref[app]{Appendix}, we
need to know that the random field ($t$ here is fixed)
\[
f ( x ) =\int_{0}^{t}\sum
_{k=0}^{\infty}\sigma_{k} \bigl( X ( s,x )
\bigr) \,dB_{s}^{k}%
\]
is continuous, or it has a continuous modification. This is true
because by
the BDG inequality,
\begin{eqnarray*}
E \bigl[ \bigl\llvert f ( x ) -f ( y ) \bigr\rrvert ^{p} \bigr] & =&E
\Biggl[ \Biggl\llvert \int_{0}^{t}\sum
_{k=0}^{\infty
} \bigl( \sigma_{k} \bigl( X (
s,x ) \bigr) -\sigma_{k} \bigl( X ( s,y ) \bigr) \bigr)
\,dB_{s}^{k}\Biggr\rrvert ^{p} \Biggr]
\\
& \leq& C_{p}E \Biggl[ \Biggl( \int_{0}^{t}
\sum_{k=0}^{\infty}\bigl\llvert
\sigma_{k} \bigl( X ( s,x ) \bigr) -\sigma_{k} \bigl( X (
s,y ) \bigr) \bigr\rrvert ^{2}\,ds \Biggr) ^{p/2} \Biggr]
\\
& \leq& C_{p}L_{\sigma}^{p}E \biggl[ \biggl( \int
_{0}^{t}\bigl\llvert X ( s,x ) -X ( s,y ) \bigr
\rrvert ^{2}\,ds \biggr) ^{p/2} \biggr]
\\
& \leq& C_{p,T}C_{p}L_{\sigma}^{p}T\llvert
x-y\rrvert ^{p}.%
\end{eqnarray*}
This last inequality follows from Lemma~\ref{limsol}. Thus, for $p>d$
we may apply Kolmogorov regularity theorem and
deduce that $f$ has a continuous version.

We show now that $\mu_t$ satisfies the conditions of Definition~\ref{def1}:
\begin{longlist}[(ii)]
\item[(i)] to prove that $ \langle\mu_{t},\phi \rangle$ is
continuous and
adapted, it is sufficient to notice that
\[
\langle\mu_{t},\phi \rangle=\int_{\mathbb{R}^{d}}\phi \bigl(
X ( t,x ) \bigr) \mu_{0} ( dx ).
\]

\item[(ii)] Let $\phi\in C_b^2(\R^d)$, we apply It\^{o}'s formula
\[
d \phi(X_t) = \nabla\phi(X_t)\cdot d X_t +
\frac{1}2\sum_k^\infty \sum
_{i,j=1}^d\partial^2_{i,j}
\phi(X_t)\sigma_k^i(X_t)
\sigma_k^j(X_t) \,d t.
\]
Under the homogeneity assumption over $\sigma_k$, we obtain the following:
\[
d \phi(X_t) = \biggl[\nabla\phi(X_t)\cdot
b(X_t) + \frac{1}2\Delta \phi (X_t) \biggr] \,d t +
\sum^\infty_k\nabla\phi(X_t)
\sigma_k(X_t) \,d B_t^k.
\]
Integrating now over $\mu_0$, we get
\[
d \langle\mu_t,\phi\rangle= \biggl[\langle\mu_t,\nabla
\phi\cdot b\rangle+ \frac{1}2\langle\mu_t,\Delta\phi\rangle
\biggr] \,d t + \sum^\infty_k\int
_{\R^d}\nabla\phi(X_t)\sigma_k(X_t)
\,d B_t^k\, d \mu_0.
\]
Using the stochastic Fubini's theorem, we interchange the stochastic
integral and the integral in $\mu_0$ and we obtain the desired equation.\quad\qed
\end{longlist}
%
\noqed\end{pf}

\subsection{The contraction mapping}\label{contraction_subsection}
In this section, we will construct a solution of equation \eqref{SPDE
Ito form} by means of a fixed-point argument. Given $\mu_0:\Omega\to
\mathcal{P}_1(\R^d)$ as in Hypothesis \ref{Initial Condition}, we
define now an operator $\Phi_{\mu_0}:\mathcal{S}\to\mathcal{S}$. In
Theorem~\ref{Thm 17}, we prove that it is a contraction and we see that
his unique fixed point is a solution to \eqref{SPDE Ito form}.

Let $\mu=(\mu_t)_{t\in[0,T]}\in\mathcal{S}$. We define the following
as the convolution between $\mu_t$ and $K$:
\[
b_\mu(t,x,\omega):=\int_{\R^d}K(x-y)
\mu_t(\omega, d y).
\]
Notice that $b_\mu(t,\cdot,\omega)$ is Lipschitz continuous with Lipschitz
constant $L_K$, which is the Lipschitz constant of $K$ and does not
depend on $t$ and $\omega$. Moreover, since $\llvert  K (
x
) \rrvert \leq L_{K} (
K ( 0 ) +\llvert  x\rrvert  ) $,
\begin{eqnarray*}
\bigl\llvert b_{\mu} ( t,0,\omega ) \bigr\rrvert & \leq &\int
_{\mathbb{R}^{d}}\bigl\llvert K ( -y ) \bigr\rrvert \mu _{t} (
\omega,dy ) \leq L_{K}\int_{\mathbb{R}^{d}} \bigl( K ( 0 ) +
\llvert y\rrvert \bigr) \mu_{t} ( \omega,dy )
\\
& \leq& L_{K}K ( 0 ) +L_{K}\int_{\mathbb{R}^{d}}
\llvert x\rrvert \mu_{t} ( \omega,dx )
\end{eqnarray*}
and the random variable $\int_{\mathbb{R}^{d}}\llvert  x\rrvert
\mu
_{t} ( \omega,dx ) $ is integrable. Hence, $b_{\mu}$
satisfies the
assumptions required in Section~\ref{Well_SPDE} to have strong existence and
uniqueness of solutions.
Let now $X^\mu_t$ be the solution to equation \eqref{liouville} with
drift coefficient $b_\mu$, namely
%
\begin{eqnarray}
\label{strong} %
d X_t &=&
b_\mu(X_t) \,d t + \sum_k
\sigma_k(X_t) \,d B_t^k,
\nonumber
\\[-8pt]
\\[-8pt]
\nonumber
X_0&=&x.
\end{eqnarray}
Let $X^\mu(t,x,\omega)$ be a modification of $X^\mu_t$ continuous in
$x$. We define, for every $t$,
%
\begin{equation}
\label{contr} (\Phi_{\mu_0} \mu)_t(\omega):=X^\mu(t,\cdot,
\omega)_\#\mu_0(\omega), \qquad\omega\mbox{-a.s.}
\end{equation}

\begin{remark}
Notice that the range of $\Phi_{\mu_0}$ is included in $\mathcal
{S}_{\mu
_0}$ and that $\Phi_{\mu_0}\mu$ is a solution of equation \eqref{SPDE
Ito form} in the sense of Definition~\ref{def1}, thanks to Proposition~\ref{propuno}.
\end{remark}

From Lemma~\ref{existence} and Proposition~\ref{propuno}, we deduce the
following theorem, which is the main result of this section.

\begin{theorem}\label{Thm 17}
Given $T>0$, the operator $\Phi_{\mu_0}$ has a unique fixed point
$\mu
=\{\mu_t\}_{t\in[0,T]}$ in $\mathcal{S}_{\mu_0}$. This fixed point
is a
solution of equation \eqref{SPDE Ito form}.
%
%
\end{theorem}
\begin{pf}
From Lemma~\ref{existence}, we have
\[
d_{\mathcal{S}}(\Phi_{\mu_0} \mu,\Phi_{\mu_0} \nu)\leq
\gamma_T d_{\mathcal{S}}(\mu,\nu)\qquad \forall \mu,\nu\in\mathcal{S},
\]
where $\gamma_T$ is defined in \eqref{gamma} as $\gamma
_T:=L_KTC_{1,T}$. Hence, there exists a time $t^*$ up to which the
operator $\Phi_{\mu_0}$ is a contraction, thus it has a unique fixed
point $\mu=(\mu_t)_{t\in[0,t^*]}$. It follows from Proposition~\ref
{propuno} that $\mu$ is a solution, in the sense of Definition~\ref
{def1}, to equation \eqref{SPDE Ito form} on the interval $[0,t^*]$,
starting from $\mu_0$. We can repeat this method on the interval
$[t^*,2t^*]$ with initial condition $\mu_{t^*}$, and iterate it up to
any finite time $T$ because $t^*$ depends only on the Lipschitz
constants of the coefficients, and not on the initial condition. In
this way, we have shown that we can construct a solution $\mu$ on the
interval $[0,T]$ which is a fixed point for the operator $\Phi_{\mu
_{nt^*}}$ on the interval $[nt^*,(n+1)t^*]$, for every $n\in\N$ such
that $nt^*<T$. Moreover, we can prove that any two fixed points $\mu
,\nu
$ of the map $\Phi_{\mu_0}$ on the interval $[0,T]$ coincide. Indeed,
if $t_0\in[0,T]$ is the largest time such that $\mu=\nu$, one proves
$t_0=T$ by contradiction, by applying the contraction argument on
$[t_0,t_0+\delta]$ for a suitable $\delta>0$, if $t_0<T$.
\end{pf}

\begin{lemma}\label{drift diverso}
Set $T>0$. Let $\mu=\{\mu_t\}_{t\geq0}, \nu=\{\nu_t\}_{t\geq0}
\in
\mathcal{S}$ and let $X^\mu, X^\nu$ be the solutions of equation
\eqref
{strong} with drift coefficients $b_\mu$ and $b_\nu$, respectively. The
following holds true:
\[
\E \Bigl[\sup_{t\in[0,T]}\bigl\llvert X^\mu(t,x)-X^\nu(t,x)
\bigr\rrvert \big|\mathcal {F}_0 \Bigr] \leq\gamma_T\E
\Bigl[\sup_{t\in[0,T]}W_1(\mu_t,\nu
_t) \big|\mathcal{F}_0 \Bigr],
\]
where
%
\begin{equation}
\label{gamma} \gamma_T:=L_KTC_{1,T}.
\end{equation}
The constant $C_{1,T}$ is defined in \eqref{second_constant}.
\end{lemma}
\begin{pf}
Given $T>0$, we call $n$ the smallest positive integer such that
$C
(1,(T/n) )<1$ [see \eqref{first_constant}].
We split the interval $[0,T]$ in $n$ subintervals, namely $[\frac
{m-1}{n} T,\frac{m}n T]$, for $m\leq n$. We will give the proof by
induction over $m$.

First, we prove our claim on the interval $[0,(T/n)]$.
We start our estimation by giving bounds for the drift and the noise of
equation \eqref{strong}. It holds, $\mathbb{P}$-a.s.,
%
\begin{eqnarray}
\label{primterm} &&\int_0^t\bigl|b_\mu
\bigl(s,X^\mu(s,x)\bigr)-b_\nu\bigl(s,X^\nu(s,x)
\bigr)\bigr| \,d s \nonumber\\
&&\qquad\leq \int_0^t\bigl|b_\mu
\bigl(s,X^\mu(s,x)\bigr)-b_\mu\bigl(s,X^\nu(s,x)
\bigr)\bigr| \,d s
\\
&&\qquad\quad{}+ \int_0^t\bigl|b_\mu
\bigl(s,X^\nu(s,x)\bigr)-b_\nu\bigl(s,X^\nu(s,x)
\bigr)\bigr| \,d s\nonumber
\\
&&\qquad\leq L_K\int_0^t
\bigl|X^\mu(s,x)-X^\nu(s,x)\bigr| \,d s+ L_K\int
_0^tW_1(\mu _s,\nu
_s) \,d s.\nonumber
\end{eqnarray}
Here, we used that, for every $t\in[0,(T/n)]$, $x\in\R^d$ and
$\mathbb{P}$-a.s.,
%
\begin{equation}
\label{trucco} \bigl|b_\mu(t,x)-b_\nu(t,x)\bigr|\leq L_K
W_1(\mu_t,\nu_t).
\end{equation}
To prove this, we apply first the definition of $b_\mu$:
\begin{eqnarray*}
&&\bigl|b_\mu(t,x)-b_\nu(t,x)\bigr|\\
&&\qquad=\biggl\llvert \int
_{\R^d}K(x-y) \,d \mu_t(y)-\int_{\R
^d}K
\bigl(x-y'\bigr) \,d \nu_t\bigl(y'\bigr)
\biggr\rrvert .
\end{eqnarray*}
Given $\omega\in\Omega$ a.s. and $t\in[0,(T/n)]$ for every $m\in
\Gamma
(\mu_t(\omega),\nu_t(\omega))$ so we can rewrite the right-hand
side as
follows and then apply the Lipshitz continuity of $K$ to obtain, for
$\mathbb{P}$-a.e. $\omega$,
\begin{eqnarray*}
&&\bigl|b_\mu(s,x)-b_\nu(s,x)\bigr|\\
&&\qquad=\biggl\llvert \int
_{\R^d\times\R^d}K(x-y) \,d m\bigl(y,y'\bigr)-\int
_{\R^d\times\R^d}K\bigl(x-y'\bigr) \,d m\bigl(y,y'
\bigr)\biggr\rrvert
\\
&&\qquad\leq\int_{\R^d\times\R^d} \bigl|K(x-y)-K\bigl(x-y'\bigr)\bigr| \,d m
\bigl(y,y'\bigr)
\\
&&\qquad \leq L_K \int_{\R^d\times\R^d} \bigl|y-y'\bigr| \,d m
\bigl(y,y'\bigr).
\end{eqnarray*}
Now \eqref{trucco} follows since $m$ is arbitrary.

Using the conditional Burkholder--Davis--Gundy inequality (see
Proposition~\ref{CBDG}) and the Lipschitz continuity of the noise, we
can estimate the following:
%
\begin{eqnarray}
\label{secterm}&& \E \biggl[\sup_{t\in[0,(T/n)]}\biggl\llvert \int
_0^t\sum_k
\sigma_k\bigl(X^\mu (s,x)\bigr)-\sigma_k
\bigl(X^\nu(s,x)\bigr) \,d B_s^k\biggr\rrvert \Big|
\mathcal{F}_0 \biggr]
\nonumber
\\
&&\qquad\leq C_1\E \biggl[ \biggl(\int_0^{(T/n)}
\sum_k \bigl(\sigma_k
\bigl(X^\mu (s,x)\bigr)-\sigma_k\bigl(X^\nu(s,x)
\bigr) \bigr)^2 \,d s \biggr)^{{1}/2} \Big|\mathcal
{F}_0 \biggr]
\nonumber
\\
&&\qquad\leq C_1L_\sigma\E \biggl[ \biggl(\int
_0^{(T/n)}\bigl\llvert X^\mu
(t,x)-X^\nu (t,x)\bigr\rrvert ^2 \,d t \biggr)^{{1}/2}
\Big|\mathcal{F}_0 \biggr]
\\
&&\qquad\leq C_1(T/n)^{{1}/2}L_\sigma\E \Bigl[ \Bigl(\sup
_{t\in
[0,(T/n)]}\bigl|X^\mu (t,x)-X^\nu(t,x)\bigr|^2
\Bigr)^{{1}/2} \big|\mathcal{F}_0 \Bigr]
\nonumber
\\
&&\qquad\leq C_1(T/n)^{{1}/2}L_\sigma\E \Bigl[\sup
_{t\in[0,(T/n)]}\bigl\llvert X^\mu (t,x)-X^\nu(t,x)
\bigr\rrvert \big|\mathcal{F}_0 \Bigr].\nonumber
\end{eqnarray}

We now use \eqref{primterm} and \eqref{secterm} to estimate the following:
\begin{eqnarray*}
&&\E \Bigl[\sup_{t\in[0,(T/n)]}\bigl\llvert X^\mu(t,x)-X^\nu(t,x)
\bigr\rrvert \big|\mathcal{F}_0 \Bigr] \\
&&\qquad\leq \E \biggl[\sup
_{t\in[0,(T/n)]}\int_0^t\bigl|b_\mu
\bigl(s,X^\mu(t,x)\bigr)-b_\nu\bigl(s,X^\nu(t,x)
\bigr)\bigr| \,d s \Big|\mathcal{F}_0 \biggr]
\nonumber
\\
&&\qquad\quad{}+ \E \biggl[\sup_{t\in[0,(T/n)]}\biggl\llvert \int
_0^t \sum_k
\bigl(\sigma _k\bigl(X^\mu (s,x)\bigr)-\sigma_k
\bigl(X^\nu(s,x)\bigr)\bigr) \,d B_s^k\biggr
\rrvert \Big|\mathcal {F}_0 \biggr]
\nonumber
\\
&&\qquad\leq \bigl(L_K (T/n)+C_1L_\sigma(T/n)^{{1}/2}
\bigr) \E \Bigl[\sup_{t\in[0,T]}\bigl\llvert X^\mu(s,x)-X^\nu(s,x)
\bigr\rrvert \big|\mathcal {F}_0 \Bigr]
\nonumber
\\
&&\qquad\quad{}+L_K (T/n)\E \Bigl[\sup_{t\in[0,(T/n)]}W_1(
\mu_t,\nu_t) \big|\mathcal {F}_0 \Bigr].
\nonumber
\end{eqnarray*}
Hence,
\begin{eqnarray*}
&&\E \Bigl[\sup_{t\in[0,(T/n)]}\bigl\llvert X^\mu(t,x)-X^\nu(t,x)
\bigr\rrvert \big|\mathcal{F}_0 \Bigr] \\
&&\qquad\leq\frac{1}{1-C (1,(T/n)
)}L_K(T/n)
\E \Bigl[\sup_{t\in[0,(T/n)]}W_1(\mu_t,
\nu_t) \big|\mathcal {F}_0 \Bigr],
\end{eqnarray*}
where $C (1,(T/n) )$ is defined in \eqref{second_constant}.

We now prove the inductive step. Suppose that for some $m-1\leq n$, it holds
%
\begin{eqnarray}
\label{induzione1}&& \E \Bigl[\sup_{t\in[({(m-2)}/{n})T,({(m-1)}/{n}) T]}\bigl\llvert
X^\mu (t,x)-X^\nu(t,x)\bigr\rrvert \big|\mathcal{F}_0
\Bigr]
\nonumber
\\[-8pt]
\\[-8pt]
\nonumber
&&\qquad\leq \Biggl(L_K(T/n)\sum_{i=1}^{m-1}
\biggl( \frac{1}{1-C (1,(T/n) )} \biggr)^i \Biggr)\E \Bigl[\sup
_{t\in[0,T]}W_1(\mu_t,\nu_t) \big|
\mathcal{F}_0 \Bigr],
\end{eqnarray}
we will prove the same for $m$.
In the same way as in the first step, one can deduce
%
\begin{eqnarray}\label
{ind1}
\label{ind3} &&\E \Bigl[\sup_{t\in[({(m-1)}/{n})T,({m}/n) T]}\bigl |X^\mu
(t,x)-X^\nu (t,x) \bigr| \big|\mathcal{F}_0 \Bigr]
\nonumber
\\[-8pt]
\\[-8pt]
\nonumber
&&\qquad\leq \E \bigl[
\bigl\llvert X^\mu \bigl((m-1)T/n,x\bigr)-X^\nu
\bigl((m-1)T/n,x\bigr)\bigr\rrvert \big|\mathcal{F}_0 \bigr]
\\
&&\label{ind2}\qquad\quad{}+ \bigl(L_K (T/n)+C_1L_\sigma(T/n)^{{1}/2}
\bigr)
\nonumber
\\[-8pt]
\\[-8pt]
\nonumber
&&\qquad\quad{} \times\E \Bigl[\sup_{t\in
[({(m-1)}/{n})T,({m}/n) T]}\bigl\llvert X^\mu(s,x)-X^\nu(s,x)
\bigr\rrvert \big|\mathcal{F}_0 \Bigr]
\\
&&\qquad\quad{}+L_K (T/n)\E \Bigl[\sup_{t\in[({(m-1)}/{n})T,({m}/n) T]}W_1(
\mu _t,\nu _t) \big|\mathcal{F}_0 \Bigr].
\end{eqnarray}
Now we use the inductive hypothesis \eqref{induzione1} to estimate
\eqref{ind1}. We put \eqref{ind2} on the left-hand side and we note
that the supremum in \eqref{ind3} is less than the supremum over the
whole interval
\begin{eqnarray*}
&&C_{1,(T/n)}^{-1}\E \Bigl[\sup_{t\in[({(m-1)}/{n})T,({m}/n)
T]}\bigl
\llvert X^\mu(t,x)-X^\nu(t,x)\bigr\rrvert \big|
\mathcal{F}_0 \Bigr]
\\
&&\qquad\leq \Biggl(L_KT\sum
_{i=1}^{(m-1)} \biggl( \frac{1}{1-C (1,(T/n)
)}
\biggr)^i \Biggr)\E \Bigl[\sup_{t\in[0,T]}W_1(
\mu_t,\nu_t) \big|\mathcal {F}_0 \Bigr]
\\
&&\quad\qquad{}+L_K (T/n)\E \Bigl[\sup_{t\in[0,T]}W_1(
\mu_t,\nu_t) \big|\mathcal {F}_0 \Bigr]
\\
&&\qquad= L_K(T/n) \Biggl(\sum_{i=1}^{(m-1)}
\biggl( \frac{1}{1-C
(1,(T/n) )} \biggr)^i+1 \Biggr)\E \Bigl[\sup
_{t\in
[0,T]}W_1(\mu_t,\nu _t) \big|
\mathcal{F}_0 \Bigr].
\end{eqnarray*}
So, \eqref{induzione1} is proved for $m$.

Finally, to obtain the constant of Lemma~\ref{drift diverso} notice that $ \frac
{1}{1-C (1,(T/n) )}>1$, hence $ ( \frac{1}{1-C
(1,(T/n) )} )^{i}\leq ( \frac{1}{1-C
(1,(T/n)
)} )^{n}$ when $i\leq n$. Thus, the constant in \eqref
{induzione1}, in the case $m=n$, can be further estimate by
\begin{eqnarray*}
\Biggl(L_K(T/n)\sum_{i=1}^m
\biggl( \frac{1}{1-C (1,(T/n)
)} \biggr)^i \Biggr)&\leq&
\Biggl(L_K(T/n)\sum_{i=1}^n
\biggl( \frac
{1}{1-C (1,(T/n) )} \biggr)^i \Biggr) \\
&\leq&
L_K(T/n) n \biggl( \frac{1}{1-C (1,(T/n) )} \biggr)^n.
\end{eqnarray*}
This last term is exactly $\gamma_T$ because of the definition of
$C_{1,T}$ [see \eqref{second_constant}].
\end{pf}

\begin{lemma}\label{existence}
For every $T>0$, we have
\[
d_{\mathcal{S}}(\Phi_{\mu_0} \mu,\Phi_{\mu_0} \nu)\leq
\gamma_T d_{\mathcal{S}}(\mu,\nu) \qquad\forall \mu,\nu\in\mathcal{S},
\]
where $\gamma_T$ is defined in \eqref{gamma}.
\end{lemma}

\begin{pf}
Let $\omega\in\Omega$ and $t\in[0,T]$ be fixed. The measure
$m=(X^\mu
(t,\cdot,\omega),\break  X^\nu(t,\cdot,\omega))_\#\mu_0$ belongs to $\Gamma((\Phi
_{\mu
_0} \mu)_t(\omega),(\Phi_{\mu_0} \nu)_t(\omega))$. Indeed, for every
$A\in\mathcal{\R}^{2d}$, it holds $m(B)=\mu_0\{x\in\R^d: X^\mu
(t,x,\omega),X^\nu(t,x,\omega)\in B\}$, which implies, for every
$A\in
\mathcal{B}(\R^d)$,
\begin{eqnarray*}
m\bigl(A\times\R^d\bigr)&=&\mu_0\bigl\{x\in
\R^d:X^\mu(t,x,\omega)\in A\bigr\}\\
&=&X^\mu (t,\cdot,
\omega )_\#\mu_0(A)=(\Phi_{\mu_0} \mu)_t(\omega)
(A).
\end{eqnarray*}
In the same way, $m(\R^d\times A)=(\Phi_{\mu_0} \nu)_t(\omega)(A)$.
Thus, from the definition of the Wasserstein metric $W_1$, it is easy
to see that
\[
d_{\mathcal{S}}(\Phi_{\mu_0} \mu,\Phi_{\mu_0} \nu)\leq\E
\biggl[\sup_{t\in[0,T]}\int_{\R^d}\bigl|X^\mu(t,x)-X^\nu(t,x)\bigr|
\,d\mu_0 \biggr].
\]
From the $\mathcal{F}_0$-measurability of the initial condition $\mu_0$
and applying Proposition~\ref{ScambioCondizionale}, we have the following:
\begin{eqnarray*}
&&\E \biggl[\E \biggl[\sup_{t\in[0,T]}\int_{\R^d}\bigl|X^\mu(t,x)-X^\nu
(t,x)\bigr| \,d\mu _0 \Big|\mathcal{F}_0 \biggr] \biggr]\\
&&\qquad=\E
\biggl[ \int_{\R^d}\E \Bigl[\sup_{t\in[0,T]}\bigl|X^\mu(t,x)-X^\nu(t,x)\bigr|
\Big|\mathcal{F}_0 \Bigr] \,d\mu _0 \biggr].
\end{eqnarray*}
Now we complete the proof applying Lemma~\ref{drift diverso} as follows:
\begin{eqnarray*}
d_{\mathcal{S}}(\Phi_{\mu_0} \mu,\Phi_{\mu_0} \nu)&\leq&\E
\biggl[\int_{\R
^d}\E \Bigl[\sup_{t\in[0,T]}\bigl
\llvert X^\mu(t,x)-X^\nu(t,x)\bigr\rrvert\big |
\mathcal{F}_0 \Bigr] \,d\mu_0 \biggr]\\
& \leq&
\gamma_Td_{\mathcal
{S}}(\mu,\nu).
\end{eqnarray*}
\upqed\end{pf}

\section{Convergence and propagation of chaos}\label{sezione_convergenza}
In this section, we will show that the distance between two solutions
of \eqref{SPDE Ito form} can be estimated by the distance between the
respective initial conditions. Since we have shown in Section~\ref{Settings} that the empirical measure solves \eqref{SPDE Ito form} with
the appropriate initial condition, we will be able to deduce from \ref
{lemma_convergenza} some results of propagation of chaos.

Last we will give a review on recent quantitative results that can be
applied together with Theorem~\ref{lemma_convergenza} to obtain a more
explicit rate of convergence to approximate the solution of SPDE \eqref
{SPDE Ito form} with the solution of SDE \eqref{SDE 1}.


\begin{theorem}\label{lemma_convergenza}
Given $T>0$, let $\mu_0,\nu_0 :\Omega\to\mathcal{P}_1(\R^d)$ be
as in
Hypothesis \ref{Initial Condition}, and let $\mu\in\mathcal{S}_{\mu
_0}$, $\nu\in\mathcal{S}_{\nu_0}$ be the respective solutions of
equation \eqref{SPDE Ito form} given by the contraction method
described before, there exists a constant $\tilde C_T>0$, such that
\[
d_{\mathcal{S}}(\mu,\nu)\leq\tilde{C}_T\E\bigl[ W_1(
\mu_0,\nu_0)\bigr].
\]
\end{theorem}

\begin{pf}
Given $T>0$, we define
\[
\tilde{C}_{T}:= \biggl(\frac{1}{(1-\gamma_{(T/n)}) (1-C
(1,(T/n) ) )} \biggr)^n,
\]
where $n\in\N$ is the smallest integer such that $\gamma
_{(T/n)}=L_kTC_{1,T}< 1$; see \eqref{second_constant} for the
definition of $C_{1,T}$, and $C (1,(T/n) )<1$, defined in
\eqref{first_constant}.
We will give the proof in the case when $T$ is small enough such that
$n=1$ and we refer to the inductive procedure used in Lemma~\ref
{limsol} for the general case. Notice that under this assumption
\[
\tilde{C}_{T}:=\frac{C_{1,T}}{1-\gamma_{T}},
\]
where $C_{1,T}$ is defined in \eqref{second_constant}.

Notice that, since $\|\mu_0\|=\|\nu_0\|=1$, the Lipschitz constants of
$b_\mu$ and $b_\nu$ are the same, $L_K$. Moreover, recalling the
definition of the operator $\Phi_{\mu_0}$ (resp., $\Phi_{\nu_0}$), it
holds that its fixed point $\mu$ (resp., $\nu$) can be written as
$\mu
_t=X^\mu(t,\cdot)_\#\mu_0$ [resp., $\nu_t=X^\nu(t,\cdot)_\#\nu_0$] where
$X^\mu
(t,x,\omega)$ [resp., $X^\nu(t,x,\omega)$] is a continuous version of
the solution of equation \eqref{liouville} with drift coefficient
$b_\mu
$ (resp., $b_\nu$).
Let now $\omega$ be fixed. Notice that the infimum in the definition of
the Wasserstein metric is indeed a minimum (see \cite{AGS}, Chapter~6),
that is, there exists a measure $m(\omega)\in\Gamma(\mu_0(\omega
),\nu
_0(\omega))$ such that
%
\begin{equation}
\label{29} \int_{\R^{d}\times\R^d}\bigl|x-x'\bigr|m\bigl(\omega, d
x, d x'\bigr)= W_1\bigl(\nu _0(\omega),\mu
_0(\omega)\bigr).
\end{equation}
Moreover, the function $\omega\mapsto m(\omega)$ is $\mathcal
{F}_0$-measurable. Indeed, for every couple of measures $(\mu,\nu)\in
\mathcal{P}_1\times\mathcal{P}_1$ we can construct a measurable map
$(\mu,\nu)\mapsto m\in\Gamma_0(\mu,\nu)$ using Proposition~\ref
{measurability} in the \hyperref[app]{Appendix}, and then we can see that the function
$\omega\mapsto(\mu_0(\omega),\nu_0(\omega))\mapsto m(\omega)$ is
$\mathcal{F}_0$-measurable since it is a composition of measurable functions.
If we define $m_t(\omega)=(X^\mu(t,\cdot,\omega),X^\nu(t,\cdot,\break \omega))_\#
m(\omega)$, we get $m_t\in\Gamma((\Phi_{\mu_0}\mu)_t,(\Phi_{\nu
_0}\nu)_t)$.

As a particular case of Lemma~\ref{limsol}, we have that
%
\begin{equation}
\label{28} \E \Bigl[\sup_{t\in[0,T]}\bigl\llvert
X^\mu(t,x)-X^\mu\bigl(t,x'\bigr)\bigr\rrvert
\big|\mathcal {F}_0 \Bigr]\leq C_{1,T}\bigl|x-x'\bigr|,
\end{equation}
where $x,x'\in\R^d$ are two initial condition for equation \eqref{strong}.

In the following estimates, we use the definition of the Wasserstein
metric, the definition of $m_t$, Proposition~\ref{ScambioCondizionale},
inequality \eqref{28} and identity \eqref{29},
%
\begin{eqnarray}
\label{A}&& \E \Bigl[\sup_{t\in[0,T]}W_1\bigl((
\Phi_{\mu_0}\mu)_t,(\Phi_{\nu
_0}\mu )_t
\bigr) \Bigr]\nonumber\\
&&\qquad\leq\E \biggl[\sup_{t\in[0,T]}\int_{\R^{2d}}\bigl|x-x'\bigr|
\,d m_t\bigl(x,x'\bigr) \biggr]
\nonumber
\\
&&\qquad= \E \biggl[\E \biggl[\sup_{t\in[0,T]}\int_{\R^{2d}}
\bigl\llvert X^\mu (t,x)-X^\mu \bigl(t,x'
\bigr)\bigr\rrvert \,d m\bigl(x,x'\bigr)\Big |\mathcal{F}_0
\biggr] \biggr]
\nonumber
\\[-8pt]
\\[-8pt]
\nonumber
&&\qquad\leq\E \biggl[\int_{\R^{2d}}\E \Bigl[\sup_{t\in[0,T]}
\bigl\llvert X^\mu (t,x)-X^\mu\bigl(t,x'\bigr)
\bigr\rrvert \big|\mathcal{F}_0 \Bigr] \,d m\bigl(x,x'\bigr)
\biggr]
\\
&&\qquad\leq \E \biggl[\int_{\R^{2d}}C_{1,T}\bigl|x-x'\bigr|
\,d m\bigl(x,x'\bigr) \biggr]
\nonumber
\\
&&\qquad= C_{1,T}\E\bigl[W_1(\mu_0,
\nu_0)\bigr].\nonumber
\end{eqnarray}
%
Using now the definition of the operators $\Phi_{\mu_0}, \Phi_{\nu
_0} $
and a triangular inequality, we obtain
%
\begin{eqnarray}
\label{B} d_{\mathcal{S}}(\mu,\nu)&= & d_{\mathcal{S}}(
\Phi_{\mu_0}\mu,\Phi _{\nu
_0}\nu)
\nonumber
\\
&\leq& d_{\mathcal{S}}(\Phi_{\mu_0}\mu,\Phi_{\nu_0}\mu) +
d_{\mathcal
{S}}(\Phi_{\nu_0}\mu,\Phi_{\nu_0}\nu)
\\
&\leq& C_{1,T}\E \bigl[W_1(\mu_0,
\nu_0) \bigr] + \gamma_T d_{\mathcal
{S}}(\mu,\nu).\nonumber
\end{eqnarray}
In the last inequality, we have used \eqref{A} and Lemma~\ref
{existence}. Inequality \eqref{B} leads to
\[
d_{\mathcal{S}}(\mu,\nu)\leq\frac{C_{1,T}}{1-\gamma_T}\E\bigl[W_1(\mu
_0,\nu_0)\bigr].
\]
%
\upqed\end{pf}

Reading the proof of this theorem, one may wonder if it is really
necessary to add the complication of splitting the time interval in
subintervals. Indeed a more simple calculation can lead to a global
estimate, although it can only be obtained if the initial conditions
belong $W_2$, which is a stronger assumption. Nevertheless, we will
give now the proof in that case so that the reader can compare the two
different approaches. Moreover, if one is interested in the $W_2$ norm,
one can apply this method to other results within this paper. We are
indebted to an anonymous referee for suggesting us this idea.

At the end of this subsection, we will stress what is the difficulty
encountered using $W_1$ which prevents us to obtain a straightforward
global estimation in time.

\begin{theorem}\label{convergenza_norma_2}
Under the same assumptions of Theorem~\ref{lemma_convergenza}, suppose
that the random measures $\mu_0,\nu_0$ take values in $\mathcal
{P}_2(\R
^d)$, namely they have finite second moments. Then it holds, for all
$t\leq T$,
\[
\E \bigl[W_2^2(\mu_t,\nu_t)
\bigr]\leq4e^{4t(2tL_k^2+C_2L_\sigma
^2)}\E \bigl[W_2^2(
\mu_0,\nu_0) \bigr],
\]
where $L_K$ and $L_\sigma$ are the Lipschitz constants of the
coefficients of the system and $C_2$ is the constant appearing in
Burkholder--Davis--Gundy inequality with exponent $2$.
\end{theorem}
\begin{pf}
Proceeding as in the proof of Theorem~\ref{lemma_convergenza}, we can
find a random measure $m\in\Gamma_0(\mu_0,\nu_0)$, such that
$W_2^2(\mu
_0,\nu_0)=\int_{\R^d\times\R^d}|x-x'| \,d m(x, x')$. Moreover, it holds
\begin{eqnarray*}
\E \bigl[W_2^2(\mu_t,\nu_t)
\bigr]&\leq&\E \biggl[\int\bigl|X^\mu (t,x)-X^\nu
\bigl(t,x'\bigr)\bigr|^2\,dm \biggr]\\
&=&\E \biggl[\int\E
\bigl[\bigl|X^\mu(t,x)-X^\nu \bigl(t,x'
\bigr)\bigr|^2\Big|\mathcal{F}_0 \bigr]\,dm\bigl(x,x'
\bigr) \biggr].
\end{eqnarray*}
Hence, we proceed estimating the conditional expectation in the last
term using that $X^\mu(t,x)$ and $X^\nu(t,x')$ solve \eqref{strong} and
a parallelogram inequality,
%
\begin{eqnarray}
&&\E \bigl[\bigl|X^\mu(t,x)-X^\nu\bigl(t,x'
\bigr)\bigr|^2|\mathcal{F}_0 \bigr]
\nonumber
\\[-8pt]
\\[-8pt]
\nonumber
&&\qquad\leq 2\bigl|x-x'\bigr|^2\label{condizione_iniziale}
\\
&&\qquad\quad{}+2\E \biggl[ \biggl(\int_0^t\bigl|b_{\mu_s}
\bigl(X^\mu(s,x)\bigr)-b_{\nu_s}\bigl(X^\nu
\bigl(s,x'\bigr)\bigr)\bigr|\,ds \biggr)^2\Bigr|
\mathcal{F}_0 \biggr]\label{termine_nonlineare}
\\
&&\qquad\quad{}+2\E \biggl[ \biggl(\int_0^t\sum
_k\bigl|\sigma_k\bigl(X^\mu(s,x)\bigr)-
\sigma _k\bigl(X^\nu \bigl(s,x'\bigr)
\bigr)\bigr|\,dB_s^k \biggr)^2\Big|\mathcal{F}_0
\biggr].\label{termine_stocastico}
\end{eqnarray}
Using a Burkholder--Davis--Gundy inequality and the Lipschitz
continuity of $\sigma_k$, we can estimate \eqref{termine_stocastico}
as follows:
%
\begin{eqnarray}
\label{termine_stocastico_finale}&& 2\E \biggl[ \biggl(\int_0^t\sum
_k\bigl|\sigma_k\bigl(X^\mu(s,x)
\bigr)-\sigma _k\bigl(X^\nu \bigl(s,x'\bigr)
\bigr)\bigr|\,dB_s^k \biggr)^2\Big|\mathcal{F}_0
\biggr]
\nonumber
\\[-8pt]
\\[-8pt]
\nonumber
&&\qquad\leq2C_2L_\sigma^2\E \biggl[\int
_0^t\bigl|X^\mu(t,x)-X^\nu
\bigl(t,x'\bigr)\bigr|^2\,ds\Big|\mathcal{F}_0
\biggr].
\end{eqnarray}
To estimate \eqref{termine_nonlineare}, we first apply the Jensen
inequality, then we need to split the drift using a triangular
inequality and then use the Lipschitz continuity of $K$,
%
\begin{eqnarray}
&&2\E \biggl[ \biggl(\int_0^t\bigl|b_{\mu_s}
\bigl(X^\mu(s,x)\bigr)-b_{\nu_s}\bigl(X^\nu
\bigl(s,x'\bigr)\bigr)\bigr|\,ds \biggr)^2\Big|
\mathcal{F}_0 \biggr]\nonumber\\
&&\qquad \leq2t\E \biggl[\int_0^t\bigl|b_{\mu
_s}
\bigl(X^\mu(s,x)\bigr)-b_{\nu_s}\bigl(X^\nu
\bigl(s,x'\bigr)\bigr)\bigr|^2\,ds\Big|\mathcal{F}_0
\biggr]
\nonumber
\\
&&\qquad \leq 4t\E \biggl[\int_0^t\,ds\int\bigl|K
\bigl(X^\mu(s,x)-y\bigr)-K\bigl(X^\nu \bigl(s,x'
\bigr)-y\bigr)\bigr|^2\,d\mu _s(y)
\nonumber
\\
&&\qquad\quad{}+\biggl\llvert \int\bigl(K\bigl(X^\nu\bigl(s,x'
\bigr)-y\bigr)\nonumber\\
&&\qquad\quad{}-K\bigl(X^\nu\bigl(s,x'
\bigr)-y'\bigr)\bigr)\,d\bigl(\mu _s(y)-\nu
_s\bigl(y'\bigr)\bigr)\biggr\rrvert ^2\Big|
\mathcal{F}_0 \biggr]
\nonumber
\\
&&\qquad\leq 4tL_k^2\int_0^t
\E \bigl[\bigl|X^\mu(s,x)-X^\nu\bigl(s,x'
\bigr)\bigr|^2|\mathcal {F}_0 \bigr]\,ds\label{termine_facile}
\\
&&\qquad\quad{}+4tL_k^2\int_0^t\E
\bigl[W_2^2(\mu_s,\nu_s)|
\mathcal{F}_0 \bigr]\,ds.\label{termine_difficile}
\end{eqnarray}
We used here a property of the Wassertein metric which we already used
and proved in the proof of Lemma~\ref{drift diverso} [see \eqref
{primterm}] for $W_1$, but which can be straightforwardly readapted to $W_2$.

We now put together \eqref{condizione_iniziale}, \eqref
{termine_facile}, \eqref{termine_difficile} and \eqref
{termine_stocastico_finale} to obtain
\begin{eqnarray*}
&&\E \bigl[W_2^2(\mu_t,\nu_t)
\bigr]\\
&&\qquad\leq\E \biggl[\int\bigl|X^\mu (t,x)-X^\nu
\bigl(t,x'\bigr)\bigr|^2\,dm\bigl(x,x'\bigr)
\biggr]\\
&&\qquad\leq 2\E \bigl[W_2^2(\mu_0,
\nu_0) \bigr]
\\
&&\qquad\quad{}+4tL_k^2\int_0^t\E
\biggl[W_2^2(\mu_s,\nu_s)+
\int\bigl|X^\mu (s,x)-X^\nu \bigl(s,x'
\bigr)\bigr|^2\,dm\bigl(x,x'\bigr) \biggr]\,ds
\\
&&\qquad\quad{}+2C_2L_\sigma^2\int_0^t
\E \biggl[\int\bigl|X^\mu(s,x)-X^\nu \bigl(s,x'
\bigr)\bigr|^2\,dm\bigl(x,x'\bigr) \biggr]\,ds.
\end{eqnarray*}
Adding at the end the positive term $2C_2L_\sigma^2\int_0^t\E
[W_2^2(\mu_s,\nu_s) ]\,ds$, we can apply the Gronwall inequality
and obtain
\begin{eqnarray*}
&&\E \biggl[W_2^2(\mu_t,\nu_t)+
\int\bigl|X^\mu(t,x)-X^\nu \bigl(t,x'
\bigr)\bigr|^2\,dm\bigl(x,x'\bigr) \biggr]\\
&&\qquad\leq4e^{4t(2tL_k^2+C_2L_\sigma^2)}
\E \bigl[W_2^2(\mu_0,\nu _0)
\bigr].
\end{eqnarray*}
\upqed\end{pf}

\begin{remark}\label{Difference between L1 and L2}
Reading the proof of the previous theorem, one can be led to think that
it is possible to do the same calculations using the norm $W_1$, which
is true up to some point. In particular, following the idea of the
proof of Theorem~\ref{convergenza_norma_2} one can reach the inequality
\begin{eqnarray*}
\E \bigl[W_1(\mu_t,\nu_t) \bigr]&\leq&\E
\biggl[\int\bigl|X^\mu_t-X^\nu _t\bigr|\,dm
\biggr]\\
&\leq&\E \bigl[W_1(\mu_0,\nu_0)
\bigr]+L_k\int_0^t\E
\biggl[W_1(\mu_s,\nu _s)+
\int\bigl|X^\mu_s-X^\nu_s\bigr|\,dm \biggr]\,ds
\\
&&{}+C_1L_\sigma\E \biggl[\int \biggl(\int
_0^t\bigl|X^\mu_s-X^\nu
_s\bigr|^2\,ds \biggr)^{{1}/2}\,dm \biggr].
\end{eqnarray*}
The difficult term is the last one, indeed we do not see a way to get
rid of the powers or to switch them with the integrals. What we indeed
do in most of the proofs in this paper is to take the supremum in time
inside the integrals to obtain
\[
\E \biggl[\int \biggl(\int_0^t\bigl|X^\mu_s-X^\nu_s\bigr|^2\,ds
\biggr)^{{1}/2}\,dm \biggr]\leq tC_1L_\sigma\E
\biggl[\int\sup_{s\in[0,t]}\bigl|X^\mu_s-X^\nu
_s\bigr|\,dm \biggr].
\]
At this point, it is no longer possible to apply the Gronwall lemma,
but this last term can be subtracted in both sides of the estimations
to get something of the form $(1-tC_1L_\sigma)\E [\int\sup_{s\in
[0,t]}|X^\mu_s-X^\nu_s|\,dm ]\leq\cdots$, from which the need to do
the estimations in small intervals first.
\end{remark}

\subsection{Propagation of chaos}
Let $(\Omega, \mathcal{F},\mathcal{F}_t, \mathbb{P})$ be a filtered
probability space, and $(X^i_0)_{i\in\mathbb{N}}$ be a sequence of
symmetric $\R^d$-valued random variable on this space that are
measurable with respect to $\mathcal{F}_0$. We consider a collection
$B^k_t$, $k\geq1$, of independent Brownian motions on this space,
independent from the $X^i_0$, and we call $(\mathcal{F}^B_t)_{t\geq0}$
the filtration generated by $(B_t^k)_{k\geq1}$. 
For every $N\in\mathbb{N}$, $X^N=(X_t^{1,N},\ldots,X_t^{N,N})_{t\geq0}$ is
the solution of equation \eqref{SDE 1} with initial condition
$(X^1_0,\ldots,X^N_0)$. We will further suppose that the empirical measure
$S_0^N:=\frac{1}N\sum_{i=0}^N\delta_{X^i_0}$ converges to a random
probability measure $\mu_0$, in the metric $\E [W_1(\cdot,\cdot) ]$.
Under these settings, we will now prove Theorems~\ref{main theorem plus}
(which is slightly more general then Theorem~\ref{main theorem}) and
\ref{ConvergenzaL1}, but first we need the following lemma.

\begin{lemma}\label{simmetria}
Let $\sigma:\{1,\ldots,N\}\to\{1,\ldots,N\}$ be a permutation. Then
%
\begin{equation}
\label{simmetriacondizionale} \E \bigl[f\bigl(X_t^{1,N},
\ldots,X_t^{N,N}\bigr)|\mathcal{F}_{t}^B
\bigr]=\E \bigl[f\bigl(X_t^{\sigma(i),N},\ldots,X_t^{\sigma(N),N}
\bigr)|\mathcal{F}_{t}^B \bigr],
\end{equation}
for every $f\in C_b ((\R^d)^N )$.
\end{lemma}
\begin{pf}
Let $X^{\sigma,N}:=(X_t^{\sigma(1),N},\ldots,X_t^{\sigma(N),N})_{t\geq0}$.
Since $X^N$ is a strong solution of equation $\eqref{SDE 1}$ with
initial condition $(X_1,\ldots,X_N)$ it is easy to see that $X^{\sigma,N}$
is a strong solution of equation $\eqref{SDE 1}$ with initial condition
$(X_{\sigma(1)},\ldots,X_{\sigma(N)})$. Since the coefficients $b$ and
$\sigma_k$ have the necessary Lipschitz properties (see \cite{Ku90}),
we have strong uniqueness at fixed initial data $x\in\R^d$. Thus, we
can apply Proposition~1.4 of \cite{revuz1999continuous} (notice that
$X^N$ and $X^{\sigma,N}$ have the same initial law) and we obtain
uniqueness in law. More precisely we have
\[
\bigl(X_t^{N},\bigl(B_t^k
\bigr)_{k\in\N}\bigr)_\#\mathbb{P}=\bigl(X_t^{\sigma
,N},
\bigl(B_t^k\bigr)_{k\in\N
}\bigr)_\#\mathbb{P}\qquad
\forall t\geq0.
\]
This implies, for every $A\in\mathcal{F}^B_t$ such that $A=\{
(B_t^k)_{k\geq1}\in\tilde A\}$ with $\tilde A\in\mathcal{B}((\R
^d)^\infty)$ and for every $\phi\in C_b((\R^d)^N)$,
\begin{eqnarray*}
\E \bigl[\mathds{1}_Af\bigl(X^N_t\bigr)
\bigr]&=&\E \bigl[\mathds{1}_{\{
(B_t^k)_{k\geq
1}\in\tilde A\}}f\bigl(X^N_t
\bigr) \bigr]=\E \bigl[\mathds{1}_{\{
(B_t^k)_{k\geq
1}\in\tilde A\}}f\bigl(X^{N,\sigma}_t
\bigr) \bigr]\\
&=&\E \bigl[\mathds {1}_Af\bigl(X^{N,\sigma}_t
\bigr) \bigr].
\end{eqnarray*}
Since the integrals of $f(X_t^N)$ and $f(X_t^{N,\sigma})$ coincide on
every element of a basis of $\mathcal{F}^B_t$, their conditional
expectation coincide also; hence, \eqref{simmetriacondizionale} follows.
\end{pf}

Using the previous result, we can now prove Theorem~\ref{teorema_2}
which we restate here for simplicity.

\begin{theorem}\label{main theorem plus}
There exists a random measure-valued solution $\mu_{t}$ of
equation~(\ref{SPDE}) such that
\[
\lim_{N\rightarrow\infty}E \bigl[ \bigl\llvert \bigl\langle
S_{t}^{N} 
,\phi \bigr\rangle- \langle
\mu_{t},\phi \rangle\bigr\rrvert \bigr] =0
\]
for all $\phi\in C_{b} ( \mathbb{R}^{d} ) $.

Moreover, given $r\in\mathbb{N}$ and $\phi_{1},\ldots,\phi_{r}\in
C_{b} (
\mathbb{R}^{d} ) $, we have
\[
\lim_{N\rightarrow\infty}E \bigl[ \phi_{1} \bigl(
X_{t}^{1,N} \bigr)\cdots\phi_{r} \bigl(
X_{t}^{r,N} \bigr) |\mathcal {F}_{t}^{B}
\bigr] =%
\E \Biggl[{ \prod
_{i=1}^{r}} 
 \langle\mu_{t},
\phi_{i} \rangle \Big|\mathcal {F}_{t}^{B} \Biggr]
\]
in $L^{1} ( \Omega ) $.
\end{theorem}

\begin{pf}
Since the convergence in the Wasserstein metric $W_1$ implies the weak
convergence, the first statement follows from Theorem~\ref{lemma_convergenza}.

Without loss of generality, we prove the second statement in the case
$r=2$. Let $\phi_1,\phi_2\leq M$. By a triangular inequality, we obtain
%
\begin{eqnarray}
\label{eq49}&&\bigl |\E \bigl[\phi_1\bigl(X_t^{1,N}
\bigr)\phi_2\bigl(X_t^{2,N}\bigr) |\mathcal
{F}_{t}^B \bigr]-\E \bigl[\langle\mu_t,
\phi_1\rangle \langle\mu _t,\phi _2\rangle |
\mathcal{F}_t^B \bigr] \bigr|
\nonumber
\\
&&\qquad\leq\bigl\llvert \E \bigl[\phi_1\bigl(X_t^{1,N}
\bigr)\phi_2\bigl(X_t^{2,N}\bigr) |\mathcal
{F}_{t}^B \bigr]-\E \bigl[\bigl\langle
S_t^N,\phi_1\bigr\rangle \bigl\langle
S_t^N,\phi _2\bigr\rangle |
\mathcal{F}_t^B \bigr]\bigr\rrvert \label{eq48}
\\
&&\qquad\quad{}+\bigl\llvert \E \bigl[\bigl\langle S_t^N,
\phi_1\bigr\rangle \bigl\langle S_t^N,\phi
_2\bigr\rangle |\mathcal{F}_t^B \bigr]-\E
\bigl[\langle\mu_t,\phi_1\rangle \langle
\mu_t,\phi_2\rangle |\mathcal{F}_t^B
\bigr]\bigr\rrvert .
\end{eqnarray}
Using Lemma~\ref{simmetria}, we can estimate \eqref{eq48} as follows:
\begin{eqnarray*}
&&\bigl\llvert \E \bigl[\phi_1\bigl(X_t^{1,N}
\bigr)\phi_2\bigl(X_t^{2,N}\bigr) |\mathcal
{F}_{t}^B \bigr]-\E \bigl[\bigl\langle
S_t^N,\phi_1\bigr\rangle \bigl\langle
S_t^N,\phi _2\bigr\rangle |
\mathcal{F}_t^B \bigr]\bigr\rrvert
\\
&&\qquad= \Biggl\llvert \frac{1}{N^2-N}\sum_{i,j=1,i\neq j}^N
\E \bigl[\phi _1\bigl(X_t^{i,N}\bigr)
\phi_2\bigl(X_t^{j,N}\bigr) |
\mathcal{F}_{t}^B \bigr]\\
&&\qquad\quad{}-\frac
{1}{N^2}\sum
_{i,j=1}^N\E \bigl[\phi_1
\bigl(X_t^{i,N}\bigr)\phi _2\bigl(X_t^{j,N}
\bigr) |\mathcal{F}_{t}^B \bigr]\Biggr\rrvert
\\
&&\qquad\leq\biggl\llvert \biggl(\frac{1}{N^2-N}-\frac{1}{N^2} \biggr)
\bigl(N^2-N\bigr)M^2\biggr\rrvert +\biggl\llvert
\frac{1}NM^2\biggr\rrvert\\
&&\qquad =2\frac{M^2}{N}\to0\qquad \mbox{as }
N\to \infty.
\end{eqnarray*}
The convergence to zero of \eqref{eq49} follows from the first
statement of this theorem. Indeed,
\begin{eqnarray*}
&&\E \bigl[\bigl\llvert \E \bigl[\bigl\langle S_t^N,
\phi_1\bigr\rangle \bigl\langle S_t^N,\phi
_2\bigr\rangle |\mathcal{F}_t^B \bigr]-\E
\bigl[\langle\mu_t,\phi _1\rangle \langle
\mu_t,\phi_2\rangle |\mathcal{F}_t^B
\bigr]\bigr\rrvert \bigr]
\\
&&\qquad\leq\E \bigl[\bigl\llvert \bigl\langle S_t^N,
\phi_1\bigr\rangle-\langle\mu _t,\phi _1
\rangle\bigr\rrvert \bigl\llvert \bigl\langle S_t^N,
\phi_2\bigr\rangle\bigr\rrvert \bigr]+\E \bigl[\bigl\llvert \bigl
\langle S_t^N,\phi_2\bigr\rangle-\langle
\mu_t,\phi_2\rangle \bigr\rrvert \bigl\llvert \langle
\mu_t,\phi_1\rangle\bigr\rrvert \bigr]
\\
&&\qquad\leq M\E \bigl[\bigl\llvert \bigl\langle S_t^N,
\phi_1\bigr\rangle-\langle\mu _t,\phi _1
\rangle\bigr\rrvert \bigr]+M\E \bigl[\bigl\llvert \bigl\langle
S_t^N,\phi _2\bigr\rangle -\langle
\mu_t,\phi_2\rangle\bigr\rrvert \bigr]
\\
&&\qquad= 2M\E \bigl[\bigl\llvert \bigl\langle S_t^N,
\phi_1\bigr\rangle-\langle\mu_t,\phi _1\rangle
\bigr\rrvert \bigr]\to0 \qquad\mbox{as } N\to\infty.
\end{eqnarray*}
\upqed\end{pf}

\begin{pf*}{Proof of Theorem~\ref{ConvergenzaL1}}
First, notice that $X^{r,N}$ is the strong solution of equation \eqref
{strong} with drift coefficient $b_\nu$, where $\nu=S^N=\{S_t^N\}
_{t\in
[0,T]}$, and initial condition $X_0^r$. We can thus write
$X^{r,N}=X^\nu
(t,X_0^r(\omega),\omega)$.

If we apply Lemma~\ref{drift diverso}, we obtain
\[
E \bigl[ \bigl\llvert X_{t}^{r,N}-X_{t}\bigr
\rrvert \bigr] \leq\gamma_Td_{\mathcal{S}}\bigl(\mu,S^N
\bigr).
\]
This last quantity goes to $0$ as $N\to\infty$ thanks to Theorem~\ref
{lemma_convergenza}.
\end{pf*}

\subsection{Quantitative estimates}\label{quantitative estimates}
As already mentioned, there are several recent results in literature
that deal with the rate of convergence of an empirical measure. In this
section, we want to give some examples of how these results can be
applied in our model using Theorem~\ref{lemma_convergenza}.
Under the assumption in the beginning of the section, we further define
$G_0^N$ the law of the initial condition $(X_0^1,\ldots,X_0^N)$ and we
denote by $G_{0,2}^N$ its first two marginals. Given a $p>0$, we
suppose that $G_0^N$ and $\mu_0$ have finite first $p$ moments
$M_p(G_0^N)$ and $M_p(\mu_0)$.

Using Theorem~2.4 of \cite{Hauray_Mischler_2014} on the initial
conditions and our estimates of Theorem~\ref{lemma_convergenza}, we can
compare the rate of convergence of the empirical measure of the
solution to the rate of convergence of just two initial particles.

\begin{corollary}
For every exponent $\gamma<(d+1+\frac{d}p)^{-1}$, there exists a finite
positive constant $\Gamma$ depending only on $p$ and $d$ such that, for
every $N\geq1$,
\[
\E\bigl[W_1\bigl(S_t^N,\mu\bigr)\bigr]\leq
\tilde C \Gamma \bigl(M_p\bigl(G_0^N
\bigr)+M_p(\mu _0) \bigr)^{{1}/p}
\biggl(W_1\bigl(G_{0,2}^N,\mu_0
\bigr)+\frac{1}N \biggr)^\gamma.
\]
\end{corollary}

When the initial condition consists of a sequence of i.i.d. $\mu
_0$-distributed random variables $(X_0^i)_{i\in\N}$, a quantitative
estimate can be derived from \cite{Fou_Gui}. Under this stronger
assumptions one can obtain a slightly stronger result, however in this
case we must suppose that the measures which we are working on have
finite $p$ moments with $p$ strictly greater than one.

\begin{corollary}
Let $p>1$. There exists a constant $\Gamma$ depending on $p$ and $d$
such that, for all $N\geq1$,
\begin{eqnarray*}
&&\E \bigl[ W_1\bigl(S^N_t,\mu_t
\bigr) \bigr]\\
&&\qquad\leq\tilde C \Gamma M_p(\mu _0)^{{1}/p}
\\
&&\qquad\quad{}\times \cases{ %
 N^{-{1}/2}\operatorname{log}(1+N)+N^{-{(p-1)}/{p}},&\quad
$\mbox{if } d=2 \mbox{ and } p\neq2,$
\vspace*{2pt}\cr
N^{-{1}/d}+N^{-{(p-1)}/{p}},&\quad $\mbox{if } d>2 \mbox{ and } p\neq
\displaystyle\frac
{d}{(d-1)}.$ }
\end{eqnarray*}
\end{corollary}

\begin{appendix}\label{app}
\section*{Appendix}
%
\begin{proposition}\label{CBDG}
Given $(\Omega, \mathcal{F}, (\mathcal{F}_t)_{t\in[0,T]},\mathbb{P})$,
let $M_t$ be a continuous martingale with respect to $\mathcal{F}_t$.
If we define $M^*_t=\sup_{0\leq s\leq t} |M_s|$, it holds
\[
\E \bigl[\bigl|M^*_t\bigr|^p|\mathcal{F}_0 \bigr]\leq
C_p \E \bigl[[M]_t^{{p}/2}|\mathcal{F}_0
\bigr],
\]
for some constant $C_p>0$.
\end{proposition}
\begin{pf}
We fix an $A\in\mathcal{F}_0$ and we prove the following:
\[
\E \bigl[\mathds{1}_A \bigl|M^*_t\bigr|^p \bigr]\leq
C_p \E \bigl[\mathds{1}_A [M]_t^{{p}/2}
\bigr].
\]

First, we note that $N_t:=M_t\mathds{1}_A$ is a continuous $\mathcal
{F}_t$-martingale, indeed $A\in\mathcal{F}_0\subset\mathcal{F}_s$ implies
\[
\E [\mathds{1}_A M_t|\mathcal{F}_s ]=
\mathds{1}_A\E [M_t|\mathcal{F}_s ]=
\mathds{1}_AM_s.
\]
We can thus apply the Burkholder--Davis--Gundy inequality to $N_t$ and
we obtain
\[
\E \bigl[\bigl|N^*_t\bigr|^p \bigr]\leq C_p \E
\bigl[[N]_t^{{p}/2} \bigr].
\]
Notice that $\mathds{1}_A$ commute with $\sup_{t\in[0,T]}$. The thesis
follows from the equality
%
\begin{equation}
\label{varquad} [\mathds{1}_A M]_t=\mathds{1}_A
[M]_t.
\end{equation}
\upqed\end{pf}
Throughout the paper, we repeatedly used an identity of the form
%
\begin{equation}
E \biggl[ \int_{\mathbb{R}^{d}}f ( x ) \,d\mu_{0} ( x ) \Big|
\mathcal{F}_{0} \biggr] =\int_{\mathbb{R}^{d}}E \bigl[ f ( x
) |\mathcal{F}_{0} \bigr] \,d\mu_{0} ( x
).\label{identity}%
\end{equation}
This identity may look at first sight completely general but it requires
appropriate assumptions of continuity in $x$ and integrability. Just in order
that all objects are well defined, we need:
\begin{longlist}[(iii)]
\item[(i)] $f:\Omega\rightarrow C ( \mathbb{R}^{d} ) $ measurable,

\item[(ii)] $E [ \int_{\mathbb{R}^{d}}\llvert  f ( x )
\rrvert
\,d\mu_{0} ( x )  ] <\infty$,

\item[(iii)] $E [ \sup_{x\in K}\llvert  f ( x ) \rrvert
 ]
<\infty$ for every compact set $K\subset\mathbb{R}^{d}$.
\end{longlist}

Indeed, under (i)--(ii), the integral $\int_{\mathbb{R}^{d}}f (
x )
\,d\mu_{0} ( x ) $ is first well defined and finite a.s. ($f$
has to
be continuous in $x$ since $\mu_{0}$ is a general probability
measure), and
also $L^{1} ( \Omega ) $, so the conditional expectation
$E [
\int_{\mathbb{R}^{d}}f ( x ) \,d\mu_{0} ( x )
|\mathcal{F}_{0} ] $ is well defined. As to the right-hand side of
(\ref{identity}), on any compact set $K\subset\mathbb{R}^{d}$, from
(i) and
(iii), we have $\omega\mapsto f ( \omega,\cdot ) $ of class
$L^{1} ( \Omega;C ( K )  ) $ [the space $C (
K ) $ of continuous functions on $K$ endowed with the uniform topology],
hence by the definition of conditional expectation of random variables with
values in Banach spaces, $E [ f|_{K}|\mathcal{F}_{0} ] $ is
again a
well-defined element of $L^{1} ( \Omega;C ( K )  )
$; and, as shown below in the proof of next proposition, taking as compact
sets the sequence of closed balls $B ( 0,n ) $ one gets a
definition of $E [ f ( x ) |\mathcal{F}_{0} ] $
as a
measurable function from $\Omega$ to $C ( \mathbb{R}^{d} ) $;
notice in particular that continuity in $x$ of $E [ f (
x )
|\mathcal{F}_{0} ] $ is essential to define $\int_{\mathbb
{R}^{d}%
}E [ f ( x ) |\mathcal{F}_{0} ] \,d\mu_{0} (
x ) $ because $\mu_{0}$ is a general probability measure.
Finally, the
finiteness of $\int_{\mathbb{R}^{d}}E [ f ( x )
|\mathcal
{F}%
_{0} ] \,d\mu_{0} ( x ) $ is ultimately a consequence of
(ii) again, as proved in the next proposition.

\begin{proposition}\label{ScambioCondizionale}
Under assumptions \textup{(i)}, \textup{(ii)} and \textup{(iii)}, identity (\ref{identity}) holds
true almost surely.
\end{proposition}

\begin{pf}
As already noticed, given $n\in\mathbb{N}$, $E [ f|_{B (
0,n )
}|\mathcal{F}_{0} ] $ is a well-defined element of $L^{1} (
\Omega;C ( B ( 0,n )  )  ) $. Moreover, if
$g$ is
in the equivalence class of $E [ f|_{B ( 0,n ) }%
|\mathcal{F}_{0} ] $, then at any $x\in B ( 0,n ) $
we have
that $g ( x ) $ is in the equivalence class of $E [
f (
x ) |\mathcal{F}_{0} ] $ [understood as the conditional
expectation of the r.v. $\omega\mapsto f ( \omega,x ) $, $x$
given]. Indeed, for every $A\in\mathcal{F}_0$,
\[
\E\bigl[g(x)\mathds{1}_A\bigr]=\E[g\mathds{1}_A](x)=
\E[f\mathds{1}_A](x)=\E \bigl[f(x)\mathds{1}_A\bigr].
\]

We can choose a sequence $f^{(m)}=\sum_{i=1}^m f_i\mathds{1}_{A_i}$
such that $f_i\in C(B ( 0,n ))$, $A_i\in\mathcal{F}$ and
$f^{(m)}\to f$ in $L^1(\Omega, C(B ( 0,n )))$, as $m\to
\infty
$. Moreover one can choose, up to subsequences, $f^{(m)}$ such that the
convergence is almost sure and $\|f^{(m)}\|_\infty\leq\|f|_{B (
0,n )}\|_\infty$, a.s. It is easy to see that $\E
[f^{(m)}|\mathcal
{F}_0]=\break \sum_i\E[f_i|\mathcal{F}_0]\mathds{1}_{A_i}$. From this it
follows that
\[
\E \biggl[\int_{B ( 0,n )}f^{(m)} \,d\mu_0 |
\mathcal {F}_0 \biggr]=\int_{B ( 0,n )}\E
\bigl[f^{(m)}\Big|\mathcal{F}_0 \bigr] \,d\mu _0,\qquad
\mathbb{P}\mbox{-a.s.}
\]
Notice that, for every fixed $\omega$, it holds $f^{(m)}(\omega)\to
f(\omega)$ uniformly in $x$ on the compact $B ( 0,n )$, and
hence, by the dominated convergence theorem
\[
\int_{B ( 0,n
)}f^{(m)}(\omega)(x)\mu_0(\omega, d x)\to\int_{B ( 0,n
)}f(\omega) (x)\mu_0(\omega, d x).
\]
Thus, $\int_{B ( 0,n
)}{f^{(n)}} \,d\mu_0\to \int_{B ( 0,n )}f \,d\mu_0$ in $L^1$ from
which follows that, up to a subsequence, $\E [\int_{K}f^{(n)}
\,d\mu
_0 | \mathcal{F}_0 ]\to\E [\int_{B ( 0,n
)}f \,d\mu
_0 |\mathcal{F}_0 ]$, $\mathbb{P}$-a.s. On the other hand, we
can first apply conditional dominated convergence and then the
traditional version of it to obtain $\int_{B ( 0,n )}\E
[f^{(n)}|\mathcal{F}_0 ] \,d\mu_0\to\int_{B ( 0,n
)}\E
[f|\mathcal{F}_0 ] \,d\mu_0$.



We have proven \eqref{identity} on a closed ball of $\R^d$, we want to
extend it on the whole space. Given $n\in\N$, we call $f_n$ the
restriction of $f$ on $B(0,n)$. It holds, as already noted, $f_n\in
L^1(\Omega,C(B(0,n)))$ for every $n\in\mathbb{N}$.

We construct now the sequence $\{g_n\}_{n\in\N}$ such that
$g_n:\Omega
\rightarrow C ( B(0,n) )$ and
\begin{eqnarray*}
g_{n} &\in&L^{1} \bigl( \Omega;C \bigl( B ( 0,n ) \bigr)
\bigr) \qquad\mbox{for every }n\in\mathbb{N},
\\
g_{n} &\in&E [ f_{n}|\mathcal{F}_{0} ]\qquad
\mbox{for every }%
n\in\mathbb{N}.
\end{eqnarray*}
We will show that there exists a function $g:\Omega\to C(\R^d)$, such
that for every $x\in\R^d$, $g(x)\in E [
f(x)|\mathcal{F}_{0} ] $ and $g|_{\Omega\times B
(0,n )
} =g_n$. Moreover, if $g,g^{\prime}:\Omega\rightarrow C (
\mathbb
{R}^{d} ) $ have
the same properties, then $g=g^{\prime}$ a.s.

First, let us prove that $g_{n+1}|_{\Omega\times B ( 0,n )
}$, as a function
from $\Omega$ to $C ( B ( 0,n )  ) $, is equal
to $g_{n}$
on a set $\Omega_{n}$ of measure one. The function $g_{n+1}$ is
characterized by two properties: it is $\mathcal{F}_{0}$-measurable,
and $E%
 [ g_{n+1}\mathds{1}_{A} ] =E [ f_{n+1}\mathds
{1}_{A}
] $ for every $A\in
\mathcal{F}_{0}$. Here, $E [ g_{n+1}\mathds{1}_{A} ] $ and
$E [
f_{n+1}\mathds{1}_{A} ] $ are elements of $C ( B (
0,n+1 )  ) $%
. Similarly, $g_{n}$ is $\mathcal{F}_{0}$-measurable, and $E [
g_{n}\mathds{1}_{A}%
 ] =E [ f_{n}\mathds{1}_{A} ] $ for every $A\in
\mathcal{F}_{0}$.
Obviously, $g_{n+1}|_{\Omega\times B ( 0,n ) }$ is
$\mathcal
{F}_{0}$%
-measurable. Moreover,
\[
E [ g_{n+1}|_{\Omega\times B ( 0,n ) }\mathds {1}_{A} ] =E [
g_{n+1}\mathds{1}_{A} ] |_{B ( 0,n ) }.
\]
To show this, notice that the function
\[
G_n(x):=\E\bigl[g_{n+1}(x)|_{\Omega\times B(0,n)}
\mathds{1}_A\bigr]
\]
is well defined by Fubini theorem as a function from $B(0,n)$ to $\R
^d$. In the same way, one can define $G(x):=\E[g_{n+1}(x)\mathds{1}_A]$
as a function on $B(0,n+1)$. Now $G_n(x)=G(x)$ for every $x\in B(0,n)$,
hence $G_n=G|_{B(0,n)}$. Now,
\begin{eqnarray*}
E [ g_{n+1}\mathds{1}_{A} ] |_{B ( 0,n ) }&=&E [
f_{n+1}\mathds{1}_{A} ] |_{B ( 0,n ) }=E [
f_{n+1}|_{\Omega\times B ( 0,n ) }\mathds{1}_{A} ] \\
&=&E [
f_{n}\mathds{1}_{A}%
 ] =E [ g_{n}
\mathds{1}_{A} ]
\end{eqnarray*}
and thus $g_{n+1}|_{\Omega\times B ( 0,n ) }$ is almost
surely equal to $g_{n}$.

On the set $\bigcap_{n}\Omega_{n}$, we have $g_{m}|_{\Omega\times
B (
0,k ) }=g_{k}$ for every $m\geq k\geq0$. Let $g:\Omega\times
\mathbb{R%
}^{d}\rightarrow\mathbb{R}$ be defined on $\bigcap_{n}\Omega_{n}$ as $%
g ( x,\omega ) =g_{m} ( x,\omega ) $ where $m$
is the
smallest integer such that $x\in B ( 0,m ) $ (and
arbitrarily on
the complementary of $\bigcap_{n}\Omega_{n}$). For every $\omega\in
\bigcap_{n}\Omega_{n,}$ the function $x\mapsto g ( x,\omega ) $ is
continuous on each $B ( 0,m ) $ (easy to check by the previous
properties). Hence, $g:\Omega\rightarrow C ( \mathbb
{R}^{d} ) $.

Now, if $g^{\prime}:\Omega\rightarrow C(\R^d)$ is such that, for every
$n\in\N$, it holds $g^{\prime}|_{\Omega\times B (0,n )
}\in\E[f_n|
\mathcal{F}_0]$, then there exists a set $\Omega_n\subset\Omega$,
such that
$\mathbb{P}(\Omega_n)=1$ and $g_n=g_n^{\prime}$ on $\Omega_n$. Then
for every
$\omega\in\bigcap_n\Omega_n$, and for every $x\in B(0,n)$, $g(\omega
,x)=g_n(\omega,x)=
g_n^\prime(\omega,x)=g^\prime(\omega,x)$; hence, $g=g^\prime$ a.e.
Finally, if $x\in B(0,n)$, and $A\in\mathcal{F}_0$,
\begin{eqnarray*}
\E\bigl[g(x)\mathds{1}_A\bigr]&=&\E\bigl[g_n(x)
\mathds{1}_A\bigr]=\E[g_n\mathds {1}_A](x)=
\E [f_n\mathds{1}_A](x)=\E\bigl[f_n(x)
\mathds{1}_A\bigr]\\
&=&\E\bigl[f(x)\mathds{1}_A\bigr].
\end{eqnarray*}
Hence, $g(x)\in\E[f(x)|\mathcal{F}_0]$.
To conclude, we notice that applying Lebesgue dominate convergence
theorem to the sequence $f_n$, the random variables $\int_{B(0,n)}f_n
\,d\mu_0$ converges a.s. to the random variable $\int_{\R^d}f \,d\mu
_0$, as
$n\to\infty$. Thus, by the conditional version of dominated
convergence theorem,
%
\begin{equation}
\label{spercond1} \E \biggl[\int_{R^d}f \,d\mu_0 \Big|
\mathcal{F}_0 \biggr]=\lim_{n\to
\infty}\E \biggl[\int
_{B(0,n)}f_n \,d\mu_0 \Big|
\mathcal{F}_0 \biggr].
\end{equation}
By the definition of $g$, we have that, as $n\to\infty$, the positive
part $g_n^+$ increases to $g^+$ a.s., and the negative $g_n^-$
increases to $g^-$. Thus, by monotone convergence theorem, it holds a.s.
%
\begin{eqnarray}
\label{spercond2} \int_{\R^d}g \,d\mu_0&=&\int
_{\R^d}g^+ \,d\mu_0-\int_{\R^d}g^-
\,d\mu _0
\nonumber
\\[-8pt]
\\[-8pt]
\nonumber
&=&\lim_{n\to\infty}\int_{B(0,n)}g_n^+
\,d\mu_0-\lim_{n\to\infty}\int_{B(0,n)}g_n^-
\,d\mu_0.
\end{eqnarray}
The thesis follows from the equalities \eqref{spercond1} and \eqref
{spercond2}. Notice that this also implies that $\int_{\R^d}\E
[f(x)|\mathcal{F}_0]\,d\mu_0(x)$ is finite, because it is equal to a
finite quantity.
\end{pf}

\begin{proposition}\label{measurability}
Let $(\mu,\nu) \in\mathcal{P}_1(\R^d)$. If we define the set
\begin{eqnarray*}
&&\Gamma_0(\mu,\nu)\\
&&\qquad:= \biggl\{\bar m\in\Gamma(\mu,\nu) \Big|\int
_{\R
^{2d}}|x-y| \,d \bar m(x,y)=\inf_{m\in\Gamma(\mu,\nu)}\int
_{\R^{2d}}|x-y| \,d m(x,y) \biggr\}
\end{eqnarray*}
then there exists a measurable function $f:\mathcal{P}_1(\R^d)\times
\mathcal{P}_1(\R^d)\to\mathcal{P}_1(\R^{2d})$ such that $f(\mu
,\nu)\in
\Gamma_0(\mu,\nu)$.
\end{proposition}
\begin{pf}
The set $\{(\mu,\nu,m)| m\in\Gamma_0(\mu,\nu)\}$ is closed in
$\mathcal
{P}_1(\R^d)\times\mathcal{P}_1(\R^d)\times\mathcal{P}_1(\R^{2d})$
endowed with the weak topology (see, e.g., \cite{AGS}, Proposition~7.1.3), thus the proposition follows from Von Neumann theorem on
measurable selections.
\end{pf}
\end{appendix}
\section*{Acknowledgments} The authors wish to thank the anonymous
referees for
the careful revision which helped to clarify and improve considerably the
initial version of this paper.







\printaddresses
\end{document}